\documentclass[oneside,reqno]{amsart}

\usepackage{amsmath,amsfonts,amsthm,amssymb}
\usepackage{bbm}
\usepackage[margin=1.3in,bottom=1in]{geometry}
\usepackage{subfigure}

\usepackage{tikz}
\usetikzlibrary{automata, positioning}
\usetikzlibrary{matrix}

\usepackage{wrapfig}
\usepackage{hyperref}

\usepackage{algorithm,algpseudocode}

\usepackage{color}

\graphicspath{{images/}}

\algrenewcommand\Return{\State \algorithmicreturn{} }

\newcommand{\bE}{\mathbb{E}}
\newcommand{\bP}{\mathbb{P}}
\newcommand{\R}{\mathbb{R}}

\newcommand{\cN}{\mathcal{N}}

\newcommand{\cP}{\mathcal{P}}

\newcommand{\N}{\mathcal{N}}
\newcommand{\E}{\mathcal{E}}
\newcommand{\cE}{\mathcal{E}}

\newcommand{\G}{\mathcal{G}}
\renewcommand{\P}{\mathcal{P}}
\newcommand{\ones}{\mathbbm{1}}
\newcommand{\one}{\mathbf{1}}
\newcommand{\eff}{{\rm eff}}

\newcommand{\cCeff}{\mathcal{C}_{\rm eff}}
\renewcommand{\hom}{\text{hom}}

\newcommand{\Ga}{\Gamma}
\newcommand{\ga}{\gamma}
\newcommand{\si}{\sigma}

\newcommand{\etol}{\epsilon_{\text{tol}}}

\DeclareMathOperator*{\argmax}{arg\,max}

\DeclareMathOperator{\supp}{supp}
\DeclareMathOperator{\Mod}{Mod}
\DeclareMathOperator{\MEO}{MEO}
\DeclareMathOperator{\Var}{Var}

\DeclareMathOperator{\Adm}{Adm}

\newtheorem{theorem}{Theorem}[section]
\newtheorem{corollary}[theorem]{Corollary}

\newtheorem{lemma}[theorem]{Lemma}

\theoremstyle{definition}
\newtheorem{definition}[theorem]{Definition}
\newtheorem{example}[theorem]{Example}

\theoremstyle{remark}
\newtheorem{remark}{Remark}[section]

\numberwithin{equation}{section}

\begin{document}

\title{Fairest edge usage and minimum expected overlap for random spanning trees}
\date{May 24, 2018}

\author{Nathan Albin$^1$}
\author{Jason Clemens$^2$}
\author{Derek Hoare}
\author{Pietro Poggi-Corradini$^1$}
\author{Brandon Sit}
\author{Sarah Tymochko}

\address{$^1$Kansas State University, Department of Mathematics, 138 Cardwell Hall, Manhattan, KS 66506}
\address{$^2$Division of Math \& Science, Missouri Valley College, Science Center, 107, Marshall, Missouri 65340}
\email[Nathan Albin]{albin@math.ksu.edu}
\email[Jason Clemens]{clemensj@moval.edu}
\email[Derek Hoare]{hoared@kenyon.edu}
\email[Pietro Poggi-Corradini]{pietro@math.ksu.edu}
\email[Brandon Sit]{sit18@up.edu}
\email[Sarah Tymochko]{sjtymo17@g.holycross.edu}
\thanks{This material is based upon work supported by the National Science Foundation under Grant Nos.~1515810 and~1659123.}

\begin{abstract}
Random spanning trees of a graph $G$ are governed by a corresponding probability mass distribution (or ``law''), $\mu$, defined on the set of all spanning trees of $G$. This paper addresses the problem of choosing $\mu$ in order to utilize the edges as ``fairly'' as possible.  This turns out to be equivalent to minimizing, with respect to $\mu$, the expected overlap of two independent random spanning trees sampled with law $\mu$.  In the process, we introduce the notion of homogeneous graphs. These are graphs for which it is possible to choose a random spanning tree so that all edges have equal usage probability.  The main result is a deflation process that identifies a hierarchical structure of arbitrary graphs in terms of homogeneous subgraphs, which we call homogeneous cores.  A key tool in the analysis is the spanning tree modulus, for which there exists an algorithm based on minimum spanning tree algorithms, such as Kruskal's or Prim's.
\end{abstract}

\maketitle

\section{Introduction}

Throughout this paper, $G=(V,E)$ will be a finite, connected multigraph with vertex set $V$ and edge set $E$.  Although we allow parallel edges, we do not allow self loops.  Often, $G$ will simply be called a ``graph'' for short.  In most of the paper, we consider unweighted graphs; notable exceptions are Sections~\ref{sec:fulkerson-duality},~\ref{sec:modulus} and~\ref{sec:weighted-graphs}.  In these sections, we will use the notation $G=(V,E,\si)$ to indicate that each edge $e\in E$ has a positive weight $\si(e)$ attached to it.

\subsection{Random spanning trees}
Let $\Ga_G$ be the set of all spanning trees of a graph $G$. Kirchhoff's matrix tree theorem shows that the cardinality of $\Ga_G$ can be computed by taking the product of all the non-zero eigenvalues of the combinatorial Laplacian matrix and dividing by $|V|$. This allows for a precise definition of {\it uniform spanning trees}, namely random spanning trees with all possibilities equally likely. Another well-known result of Kirchhoff states that the probability that a given edge belongs to a uniform spanning tree is equal to the effective resistance of that edge, when thinking of the graph as an electrical network with unit edge conductances.

 In this paper, we are interested in more general random spanning trees. To help with the terminology, whenever $\mu$ is a probability mass function (or pmf) on $\Ga_G$, we will write $\mu\in\cP(\Ga_G)$ and we will say that a random spanning tree $\underline{\ga}$ is $\mu$-random or ``has law $\mu$'' (and write $\underline{\ga}\sim \mu$), if
\begin{equation}
  \label{eq:prob-mu}
\bP_\mu(\underline{\ga}=\ga)=\mu(\ga)\qquad\forall\ga\in\Ga_G.
\end{equation}
We use the underline to distinguish the random object $\underline{\ga}$ from its possible values $\ga$.

\subsection{The minimum expected overlap problem}
To begin the discussion, we pose the following problem.  How can we
minimize the expected overlap of two independent random spanning trees with the same law? In other words, given $\mu\in\P(\Gamma_G)$, we consider two independent, identically
distributed random spanning trees
$\underline{\gamma},\underline{\gamma}'\sim\mu$.  The size of the
intersection between these two trees,
$|\underline{\gamma}\cap\underline{\gamma}'|$, is an
integer-valued random variable whose expectation is computed as follows.
\begin{equation*}
  \mathbb{E}_{\mu}|\underline{\gamma}\cap\underline{\gamma}'|
  := \sum_{\gamma,\gamma'\in\Gamma_G}|\gamma\cap\gamma'|\mu(\gamma)\mu(\gamma').
\end{equation*}
The \emph{minimum expected overlap} (MEO) problem is the problem of finding
a law $\mu$ that minimizes this expected overlap:
\begin{equation}
  \label{eq:min-overlap}
  \begin{split}
    \text{minimize}\quad&
    \mathbb{E}_{\mu}|\underline{\gamma}\cap\underline{\gamma}'|\\
    \text{subject to}\quad&\mu\in\P(\Gamma_G).
  \end{split}
\end{equation}
In other words, we seek a law for which iid random trees
``collide'' as little as possible on average.

\subsection{The fairest edge usage problem}

When a law $\mu\in\cP(\Ga_G)$ is given, we can measure the \emph{edge usage probability} (or {\it expectation}) of an edge $e\in E$, which we write as
\begin{equation}\label{eq:edge-usage-prob}
  \mathbb{P}_\mu(e\in\underline{\gamma}) :=
  \sum_{\gamma\in\Gamma_G}\mathbbm{1}_{\{e\in\gamma\}}\mu(\gamma),\qquad
  \text{where}\qquad
  \mathbbm{1}_{\{e\in\gamma\}} =
  \begin{cases}
    1 & \text{if } e\in\gamma,\\
    0 & \text{if } e\notin\gamma.
  \end{cases}
\end{equation}
This evaluates the likelihood that a particular edge $e\in E$ is used
in the random tree $\underline{\gamma}\sim\mu$.

The optimal laws $\mu^*$ for $\MEO(\Ga_G)$ are not in general unique. However, the corresponding edge usage probabilities $\eta^*(e)=\bP_{\mu^*}(e\in\underline{\ga})$ do not depend on the optimal law $\mu^*$ and, as we will see in Theorem \ref{thm:variance}, they uniquely minimize the variance of edge usage probabilities. In other words, the edge usage probabilities of any optimal pmf for $\MEO(\Ga_G)$ are given by the unique solution to the following problem:
 \begin{equation}
    \label{eq:min-variance}
    \begin{split}
      \text{minimize}\quad &
      \Var(\eta)\\
      \text{subject to}\quad & \eta(e) =
      \mathbb{P}_{\mu}\left(e\in\underline{\gamma}\right)\;
      \forall e\in E\\
      & \mu\in\P(\Gamma_G).
    \end{split}
  \end{equation}
We call this problem the \emph{fairest edge usage} (FEU) problem.

The set of optimal laws $\mu^*$ induces a natural partition of $\Gamma_G$ into two sets, which we call the \emph{fair} and \emph{forbidden} trees of $G$.  
\begin{definition}\label{def:fair-forbidden}
  Given a graph $G=(V,E)$ and its family of spanning
  trees $\Ga_G$, a spanning tree $\ga\in\Ga_G$ is called a
  \emph{fair tree} if there exists an optimal pmf $\mu^*$ such that
  $\mu^*(\ga)>0$.  The set of all fair spanning trees for $G$ is
  denoted $\Ga_G^f$.  A tree $\ga\in\Ga_G\setminus\Ga_G^f$ (if such a
  tree exists) is called a \emph{forbidden tree}.
\end{definition}
In other words, the fair trees are those trees that occur in the support of \emph{some} optimal $\mu^*$, while the forbidden trees cannot occur in the support of \emph{any} optimal $\mu^*$.  Thus, the forbidden trees of $G$ do not contribute to the value of $\eta^*$.

\subsection{Homogeneous cores and the deflation process}

In view of the variance minimization in the FEU problem (\ref{eq:min-variance}), it is natural to consider the class of graphs for which it is possible to make the variance in expected edge usage zero. We call such graphs {\it homogeneous graphs}. One of our main results, Theorem~\ref{thm:component}, can be summarized as follows.
\begin{theorem}
Let $G=(V,E)$ be a finite, connected multigraph with no self loops. Then, $G$ admits a  homogeneous core $H$, i.e., $G$ admits a connected vertex-induced subgraph $H$ with at least one edge and with the property that every fair tree $\ga\in\Ga_G$  restricts to a spanning tree of $H$.
\end{theorem}

We then use this theorem to develop a deflation process, which identifies a homogeneous core, shrinks it to a single vertex, and then repeats. This process reveals an interesting hierarchical structure of general graphs.  Our second main result, Theorem~\ref{thm:serial-spt}, can be summarized as follows.
\begin{theorem}
Let $G$ be a graph with homogeneous core $H\subsetneq G$, and let $G/H$ be the result of shrinking $H$ to a single node and pruning away any self-loops. Then a serial rule holds for the $\MEO$ problem; namely, the minimum expected overlap on $G$ is the sum of the corresponding minimum expected overlaps on $H$ and $G/H$. Moreover, any optimal pmf for $\MEO(\Ga_G)$ can be constructed by coupling any two optimal pmfs for $\MEO(\Ga_H)$ and $\MEO(\Ga_{G/H})$ respectively.
\end{theorem}
Consider, for example, the graph in Figure~\ref{fig:deflate}(a). Here, the optimal edge usage
probability function $\eta^*(e)=\mathbb{P}_{\mu^*}(e\in\underline{\gamma})$ takes
three distinct values, indicated by the edge styles as described in
the figure caption.
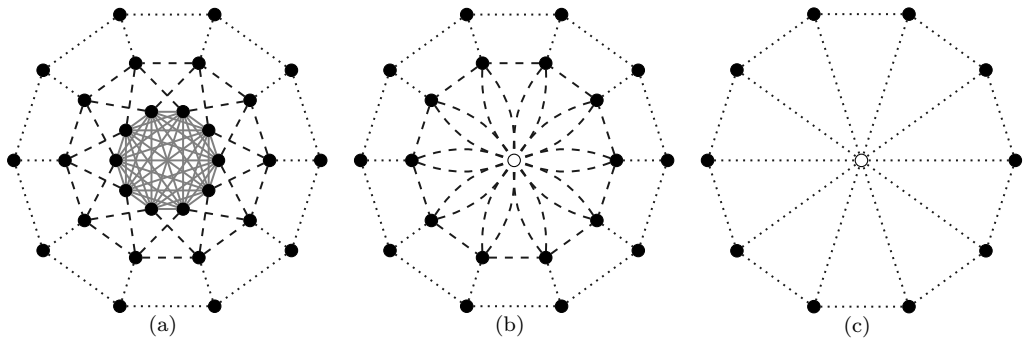
\begin{figure}
  \centering
  \subfigure[]{
    \begin{tikzpicture}[baseline,scale=0.68]
      \tikzstyle{node} = [draw,shape=circle,fill=black,scale=0.5];
      \tikzstyle{comp0} = [draw=black!50,line width=0.75pt];
      \tikzstyle{comp1} = [draw=black!90,line width=0.75pt,dashed];
      \tikzstyle{comp2} = [draw=black!90,line width=0.75pt,dotted];
      \node[node] (0) at (1.0000,0.0000) {};
      \node[node] (1) at (0.8090,0.5878) {};
      \node[node] (2) at (0.3090,0.9511) {};
      \node[node] (3) at (-0.3090,0.9511) {};
      \node[node] (4) at (-0.8090,0.5878) {};
      \node[node] (5) at (-1.0000,0.0000) {};
      \node[node] (6) at (-0.8090,-0.5878) {};
      \node[node] (7) at (-0.3090,-0.9511) {};
      \node[node] (8) at (0.3090,-0.9511) {};
      \node[node] (9) at (0.8090,-0.5878) {};
      \node[node] (10) at (2.0000,0.0000) {};
      \node[node] (11) at (1.6180,1.1756) {};
      \node[node] (12) at (0.6180,1.9021) {};
      \node[node] (13) at (-0.6180,1.9021) {};
      \node[node] (14) at (-1.6180,1.1756) {};
      \node[node] (15) at (-2.0000,0.0000) {};
      \node[node] (16) at (-1.6180,-1.1756) {};
      \node[node] (17) at (-0.6180,-1.9021) {};
      \node[node] (18) at (0.6180,-1.9021) {};
      \node[node] (19) at (1.6180,-1.1756) {};
      \node[node] (20) at (3.0000,0.0000) {};
      \node[node] (21) at (2.4271,1.7634) {};
      \node[node] (22) at (0.9271,2.8532) {};
      \node[node] (23) at (-0.9271,2.8532) {};
      \node[node] (24) at (-2.4271,1.7634) {};
      \node[node] (25) at (-3.0000,0.0000) {};
      \node[node] (26) at (-2.4271,-1.7634) {};
      \node[node] (27) at (-0.9271,-2.8532) {};
      \node[node] (28) at (0.9271,-2.8532) {};
      \node[node] (29) at (2.4271,-1.7634) {};
      \path[comp0] (0) -- (1);
      \path[comp0] (0) -- (2);
      \path[comp0] (0) -- (3);
      \path[comp0] (0) -- (4);
      \path[comp0] (0) -- (5);
      \path[comp0] (0) -- (6);
      \path[comp0] (0) -- (7);
      \path[comp0] (0) -- (8);
      \path[comp0] (0) -- (9);
      \path[comp1] (0) -- (11);
      \path[comp1] (0) -- (19);
      \path[comp0] (1) -- (2);
      \path[comp0] (1) -- (3);
      \path[comp0] (1) -- (4);
      \path[comp0] (1) -- (5);
      \path[comp0] (1) -- (6);
      \path[comp0] (1) -- (7);
      \path[comp0] (1) -- (8);
      \path[comp0] (1) -- (9);
      \path[comp1] (1) -- (10);
      \path[comp1] (1) -- (12);
      \path[comp0] (2) -- (3);
      \path[comp0] (2) -- (4);
      \path[comp0] (2) -- (5);
      \path[comp0] (2) -- (6);
      \path[comp0] (2) -- (7);
      \path[comp0] (2) -- (8);
      \path[comp0] (2) -- (9);
      \path[comp1] (2) -- (11);
      \path[comp1] (2) -- (13);
      \path[comp0] (3) -- (4);
      \path[comp0] (3) -- (5);
      \path[comp0] (3) -- (6);
      \path[comp0] (3) -- (7);
      \path[comp0] (3) -- (8);
      \path[comp0] (3) -- (9);
      \path[comp1] (3) -- (12);
      \path[comp1] (3) -- (14);
      \path[comp0] (4) -- (5);
      \path[comp0] (4) -- (6);
      \path[comp0] (4) -- (7);
      \path[comp0] (4) -- (8);
      \path[comp0] (4) -- (9);
      \path[comp1] (4) -- (13);
      \path[comp1] (4) -- (15);
      \path[comp0] (5) -- (6);
      \path[comp0] (5) -- (7);
      \path[comp0] (5) -- (8);
      \path[comp0] (5) -- (9);
      \path[comp1] (5) -- (14);
      \path[comp1] (5) -- (16);
      \path[comp0] (6) -- (7);
      \path[comp0] (6) -- (8);
      \path[comp0] (6) -- (9);
      \path[comp1] (6) -- (15);
      \path[comp1] (6) -- (17);
      \path[comp0] (7) -- (8);
      \path[comp0] (7) -- (9);
      \path[comp1] (7) -- (16);
      \path[comp1] (7) -- (18);
      \path[comp0] (8) -- (9);
      \path[comp1] (8) -- (17);
      \path[comp1] (8) -- (19);
      \path[comp1] (9) -- (10);
      \path[comp1] (9) -- (18);
      \path[comp1] (10) -- (19);
      \path[comp1] (10) -- (11);
      \path[comp2] (10) -- (20);
      \path[comp1] (11) -- (12);
      \path[comp2] (11) -- (21);
      \path[comp1] (12) -- (13);
      \path[comp2] (12) -- (22);
      \path[comp1] (13) -- (14);
      \path[comp2] (13) -- (23);
      \path[comp1] (14) -- (15);
      \path[comp2] (14) -- (24);
      \path[comp1] (15) -- (16);
      \path[comp2] (15) -- (25);
      \path[comp1] (16) -- (17);
      \path[comp2] (16) -- (26);
      \path[comp1] (17) -- (18);
      \path[comp2] (17) -- (27);
      \path[comp1] (18) -- (19);
      \path[comp2] (18) -- (28);
      \path[comp2] (19) -- (29);
      \path[comp2] (20) -- (21);
      \path[comp2] (20) -- (29);
      \path[comp2] (21) -- (22);
      \path[comp2] (22) -- (23);
      \path[comp2] (23) -- (24);
      \path[comp2] (24) -- (25);
      \path[comp2] (25) -- (26);
      \path[comp2] (26) -- (27);
      \path[comp2] (27) -- (28);
      \path[comp2] (28) -- (29);
    \end{tikzpicture}
  }
  \subfigure[]{
    \begin{tikzpicture}[baseline,scale=0.68]
      \tikzstyle{node} = [draw,shape=circle,fill=black,scale=0.5];
      \tikzstyle{vnode} = [draw,shape=circle,fill=white,scale=0.5];
      \tikzstyle{comp1} = [draw=black!90,line width=0.75pt,dashed];
      \tikzstyle{comp2} = [draw=black!90,line width=0.75pt,dotted];
      \node[node] (10) at (2.0000,0.0000) {};
      \node[node] (11) at (1.6180,1.1756) {};
      \node[node] (12) at (0.6180,1.9021) {};
      \node[node] (13) at (-0.6180,1.9021) {};
      \node[node] (14) at (-1.6180,1.1756) {};
      \node[node] (15) at (-2.0000,0.0000) {};
      \node[node] (16) at (-1.6180,-1.1756) {};
      \node[node] (17) at (-0.6180,-1.9021) {};
      \node[node] (18) at (0.6180,-1.9021) {};
      \node[node] (19) at (1.6180,-1.1756) {};
      \node[node] (20) at (3.0000,0.0000) {};
      \node[node] (21) at (2.4271,1.7634) {};
      \node[node] (22) at (0.9271,2.8532) {};
      \node[node] (23) at (-0.9271,2.8532) {};
      \node[node] (24) at (-2.4271,1.7634) {};
      \node[node] (25) at (-3.0000,0.0000) {};
      \node[node] (26) at (-2.4271,-1.7634) {};
      \node[node] (27) at (-0.9271,-2.8532) {};
      \node[node] (28) at (0.9271,-2.8532) {};
      \node[node] (29) at (2.4271,-1.7634) {};
      \node[vnode] (v1) at (0,0) {};
      \path[comp1] (10) -- (19);
      \path[comp1] (10) -- (11);
      \path[comp2] (10) -- (20);
      \path[comp1] (11) -- (12);
      \path[comp2] (11) -- (21);
      \path[comp1] (12) -- (13);
      \path[comp2] (12) -- (22);
      \path[comp1] (13) -- (14);
      \path[comp2] (13) -- (23);
      \path[comp1] (14) -- (15);
      \path[comp2] (14) -- (24);
      \path[comp1] (15) -- (16);
      \path[comp2] (15) -- (25);
      \path[comp1] (16) -- (17);
      \path[comp2] (16) -- (26);
      \path[comp1] (17) -- (18);
      \path[comp2] (17) -- (27);
      \path[comp1] (18) -- (19);
      \path[comp2] (18) -- (28);
      \path[comp2] (19) -- (29);
      \path[comp2] (20) -- (21);
      \path[comp2] (20) -- (29);
      \path[comp2] (21) -- (22);
      \path[comp2] (22) -- (23);
      \path[comp2] (23) -- (24);
      \path[comp2] (24) -- (25);
      \path[comp2] (25) -- (26);
      \path[comp2] (26) -- (27);
      \path[comp2] (27) -- (28);
      \path[comp2] (28) -- (29);
      \path[comp1] (10) edge [bend left=20] (v1);
      \path[comp1] (10) edge [bend right=20] (v1);
      \path[comp1] (11) edge [bend left=20] (v1);
      \path[comp1] (11) edge [bend right=20] (v1);
      \path[comp1] (12) edge [bend left=20] (v1);
      \path[comp1] (12) edge [bend right=20] (v1);
      \path[comp1] (13) edge [bend left=20] (v1);
      \path[comp1] (13) edge [bend right=20] (v1);
      \path[comp1] (14) edge [bend left=20] (v1);
      \path[comp1] (14) edge [bend right=20] (v1);
      \path[comp1] (15) edge [bend left=20] (v1);
      \path[comp1] (15) edge [bend right=20] (v1);
      \path[comp1] (16) edge [bend left=20] (v1);
      \path[comp1] (16) edge [bend right=20] (v1);
      \path[comp1] (17) edge [bend left=20] (v1);
      \path[comp1] (17) edge [bend right=20] (v1);
      \path[comp1] (18) edge [bend left=20] (v1);
      \path[comp1] (18) edge [bend right=20] (v1);
      \path[comp1] (19) edge [bend left=20] (v1);
      \path[comp1] (19) edge [bend right=20] (v1);
    \end{tikzpicture}
    }
    \subfigure[]{
    \begin{tikzpicture}[baseline,scale=0.682]
      \tikzstyle{node} = [draw,shape=circle,fill=black,scale=0.5];
      \tikzstyle{vnode} = [draw,shape=circle,fill=white,scale=0.5];
      \tikzstyle{comp2} = [draw=black!90,line width=0.75pt,dotted];
      \node[node] (20) at (3.0000,0.0000) {};
      \node[node] (21) at (2.4271,1.7634) {};
      \node[node] (22) at (0.9271,2.8532) {};
      \node[node] (23) at (-0.9271,2.8532) {};
      \node[node] (24) at (-2.4271,1.7634) {};
      \node[node] (25) at (-3.0000,0.0000) {};
      \node[node] (26) at (-2.4271,-1.7634) {};
      \node[node] (27) at (-0.9271,-2.8532) {};
      \node[node] (28) at (0.9271,-2.8532) {};
      \node[node] (29) at (2.4271,-1.7634) {};
      \node[vnode] (v1) at (0,0) {};
      \path[comp2] (20) -- (21);
      \path[comp2] (20) -- (29);
      \path[comp2] (21) -- (22);
      \path[comp2] (22) -- (23);
      \path[comp2] (23) -- (24);
      \path[comp2] (24) -- (25);
      \path[comp2] (25) -- (26);
      \path[comp2] (26) -- (27);
      \path[comp2] (27) -- (28);
      \path[comp2] (28) -- (29);
      \path[comp2] (20) -- (v1);
      \path[comp2] (21) -- (v1);
      \path[comp2] (22) -- (v1);
      \path[comp2] (23) -- (v1);
      \path[comp2] (24) -- (v1);
      \path[comp2] (25) -- (v1);
      \path[comp2] (26) -- (v1);
      \path[comp2] (27) -- (v1);
      \path[comp2] (28) -- (v1);
      \path[comp2] (29) -- (v1);
    \end{tikzpicture}
  }
    \caption{Hierarchical structure exposed by the solution to
      the MEO
      problem~\eqref{eq:min-overlap}. Edges are styled according to
      the edge usage probability
      $\eta(e)=\mathbb{P}_\mu(e\in\underline{\gamma})$ when $\mu$ is
      optimal. Solid edges are present with probability $1/5$, dashed
      edges with probability $1/3$ and dotted edges with probability
      $1/2$.}
    \label{fig:deflate}
\end{figure}
The subgraph induced by the least-used edges (those with probability $1/5$) form
a homogeneous core, as guaranteed by
Theorem~\ref{thm:component}.  The process of
deflation consists in shrinking this core to a
single node producing a new graph (Figure~\ref{fig:deflate}(b)) where
the core is replaced by the white node at the center; edges within the
core (which would produce self-loops) are removed; and edges from
outside the core to the core are retained with multiplicity.
Note that every edge in Figure~\ref{fig:deflate}(b) corresponds to a
unique edge in Figure~\ref{fig:deflate}(a).  The power of deflation
lies in the fact that any optimal law $\mu^*$ for Figure~\ref{fig:deflate}(b)
gives each edge the exact same edge usage probability that it had
before the homogeneous core was shrunk (see Section~\ref{ssec:deflation}).  Any graph is guaranteed to have a non-trivial homogeneous core, hence the
deflation process can be iterated, producing, in this case, the graph in
Figure~\ref{fig:deflate}(c), which is itself homogeneous, meaning that there exists a law $\mu$ giving all edges equal usage probability.

In this way, the minimum expected overlap problem  produces a type
of hierarchical deconstruction of a graph, identifying one or more
highly-connected homogeneous cores surrounded by increasingly sparse
peripheral layers.

\subsection{Families of objects and Fulkerson duality}
\label{sec:fulkerson-duality}

The key to our approach to the MEO problem is its connection to the theory of modulus of families of objects \cite{acfpc}.
Note that, using (\ref{eq:edge-usage-prob}), the unique optimal edge usage probability $\eta^*$ can be thought of as the expected indicator function of fair random trees. It turns out that $\eta^*$ can also be described as the extremal density of a modulus problem. Before explaining how this works, we recall the basic dual structure of modulus problems.

Assume $G=(V,E,\si)$ is a weighted graph with edge weights $\si\in \R_{>0}^E$.  By a \emph{family of objects} on $G$, we mean a pair $(\Ga,\cN)$.  In principle, $\Ga$ is simply a countable (possibly infinite) index set.  In practice, $\Ga$ is typically associated with certain ``real'' objects on $G$, such as paths, cuts, spanning trees, etc.  $\cN\in\mathbb{R}^{\Ga\times E}_{\ge 0}$ is called the \emph{usage matrix} for the family: each \emph{object} $\ga\in\Ga$ is given a corresponding {\it usage vector} $\cN(\ga,\cdot)^T\in\R_{\ge 0}^E$, where
\[
\text{$\cN(\ga,e):=$ the usage of $e$ by $\ga$}.
\]
The term \emph{usage} here is flexible; when $\Ga\subset 2^E$, a common choice is simply to use the indicator vector $\ones_{\{e\in\gamma\}}$ as in~\eqref{eq:edge-usage-prob}, but as we shall see, more general concepts of usage arise naturally.  To simplify notation, we typically refer to $\Ga$ alone as the family and think of $\cN$ as a ``universal'' symbol so that the notation $\cN(\ga,e)$ is always used to denote usage.

The $\MEO$ problem (\ref{eq:min-overlap}) can be generalized to any family of objects by computing the {\it weighted overlap} between two objects $\ga$ and $\ga'$ as follows.
\begin{equation}
  \label{eq:min-overlap-gen}
C(\ga,\ga'):=\sum_{e\in E}\si(e)^{-1}\cN(\ga,e)\cN(\ga',e).
\end{equation}

In what follows, we identify each object $\ga\in\Ga$ by its corresponding vector $\cN(\ga,\cdot)^T$ in the orthant $\R_{\ge 0}^E$, which we call $\mathbf{\eta}${\it -space}.
On the other hand, $\mathbf{\rho}${\it -space} is the orthant $\R_{\ge 0}^E$, where we interpret $\rho$ as a {\it cost-function} or {\it density}. Namely, for every $e\in E$,
\[
\rho(e):=\text{ the cost of using edge $e$}.
\]
(Although these two set correspond to exactly the same subset of $\mathbb{R}^E$, it can aid intuition to think of $\eta$ and $\rho$ objects of different type.  The reason they appear to reside in the same set is simply that the dual convex cone to $\mathbb{R}_{\ge 0}^E$ can be identified with $\mathbb{R}_{\ge 0}^E$.)
Given $\rho$, every object $\ga\in\Ga$ acquires a {\it total usage cost}:
\[
\ell_\rho(\ga):=\sum_{e\in E}\cN(\ga,e)\rho(e) = (\cN\rho)(e).
\]
A density $\rho$ is {\it admissible} for a family $\Ga$ if, in informal terms,  ``everyone pays at least a dollar,'' namely, if
\[
\ell_\rho(\ga)\ge 1\qquad\forall\ga\in\Ga.
\]
In terms of the usage matrix, this condition can be written as
\[
\cN\rho\ge\one,
\]
where $\one$ is the vector of all ones in $\R^\Ga$.
Let $\Adm \Ga$ be the set of all admissible densities in $\rho$-space:
\[
\Adm \Ga:=\{\rho\in\R^E_{\ge 0}: \cN\rho\ge\one\}.
\]
Note that $\Adm\Ga$ is closed, convex in $\R_{\ge 0}^E$, and recessive, meaning that adding a non-negative vector $z$ to an admissible density $\rho$ does not change admissibility. In formulas,
\[
\Adm \Ga +\R_{\ge 0}^E=\Adm\Ga.
\]
The {\it Fulkerson dual} family for $\Ga$ is:
\[
\hat{\Ga}:={\rm ext}(\Adm\Ga)\subset \R_{\ge 0}^E,
\]
where ${\rm ext}$ denotes the set of {\it extreme points}.  Since $\hat\Gamma$ is a set of points in $\rho$-space, we can interpret it as a dual family of objects to $\Ga$ which, in turn, has its own dual family.
Fulkerson duality (see \cite{acfpc}) states that:
\begin{equation}\label{eq:dual-dual}
\hat{\hat{\Ga}}\subset\Ga.
\end{equation}

\subsection{Modulus of families of objects}
\label{sec:modulus}
 Fix $1\le p<\infty$, the $p${\it -modulus} of a family of objects $\Ga$ is:
\begin{equation}\label{eq:p-modulus-def}
\Mod_{p,\si}(\Ga):=\inf_{\cN\rho\ge 1}\sum_{e\in E}\si(e)\rho(e)^p.
\end{equation}
We say that $\cE_{p,\si}(\rho):=\sum_{e\in E}\si(e)\rho(e)^p$  is  the {\it energy} of the density $\rho$.
If $p=\infty$, $\Mod_{\infty,\si}(\Ga)$ is defined using $\cE_{\infty,\si}(\rho)=\max_{e\in E}\{\si(e)\rho(e)\}$. In geometric terms, the modulus is related to the weighted $p$-norm distance from the convex set $\Adm\Ga$ to the origin in $\mathbb{R}^E$.  When $1<p<\infty$, strict convexity of the $p$-norm implies that there is a unique optimal density $\rho^*$.
For $p=1,\infty$, an optimal $\rho^*$ exists, but may not be unique.
Moreover, $1$-modulus is a linear program, so at least one optimal $\rho^*$ must occur at an extreme point of $\Adm \Ga$. Namely,
\begin{equation}\label{eq:gen-mincut}
\Mod_{1,\si}(\Ga)=\min_{\hat{\ga}\in\hat{\Ga}}\si^T \hat{\ga}.
\end{equation}
Furthermore, by varying $\si$, every $\hat{\ga}\in\hat{\Ga}$ arises as the unique optimal solution of a particular $\Mod_{1,\si}(\Ga)$ problem.
 Equation (\ref{eq:gen-mincut}) says that $1$-modulus is a generalized {\it min-cut} problem.  (See Example~\ref{ex:paths-cuts}.)

Fulkerson duality for modulus states that (see Theorem 3.7 of \cite{acfpc}): if $1<p<\infty$ and $pq=p+q$, then
\begin{equation}\label{eq:fulkerson-duality}
\Mod_{p,\si}(\Ga)^{1/p}\Mod_{q,\si^{1-q}}(\hat{\Ga})^{1/q}=1.
\end{equation}
Moreover,
\begin{equation}\label{eq:eta-star-rho-star}
\eta^*(e)=\frac{\si(e)\rho^*(e)^{p-1}}{\Mod_{p,\si}(\Ga)}\qquad\forall e\in E.
\end{equation}
A limit can be taken in (\ref{eq:fulkerson-duality}), to obtain:
\[
\Mod_{1,\si}(\Ga)\Mod_{\infty,\si^{-1}}(\hat{\Ga})=1.
\]
In the special case when $p=2$, we get that
\begin{equation}\label{eq:ptwofulkerson}
\Mod_{2,\si}(\Ga)\Mod_{2,\si^{-1}}(\hat{\Ga})=1.
\end{equation}
Moreover, in this case, the relation between the extremal density $\eta^*$ for $\Mod_{2,\si^{-1}}(\hat{\Ga})$ and the extremal density $\rho^*$ for $\Mod_{2,\si}(\Ga)$ simplifies to
\begin{equation}\label{eq:eta-star-rho-star-p2}
\eta^*(e)=\frac{\si(e)}{\Mod_{2,\si}(\Ga)}\rho^*(e)\qquad\forall e\in E,
\end{equation}
which, in the unweighted $p=2$ case, says that $\eta^*$ and $\rho^*$ are parallel.

\begin{example}[Paths and cuts]
\label{ex:paths-cuts}
When $\Ga=\Ga_s(a,b)$ is the family of all simple paths connecting nodes $a$ and $b$, the Fulkerson dual is the family $\hat{\Ga}=\Ga_c(a,b)$ of all minimal $ab$-cuts. So, by (\ref{eq:gen-mincut}), $\Mod_{1,\si}(\Ga)$ is the usual min-cut problem with capacities $\si$. Furthermore, $\Mod_{2,\si}(\Ga)$ is the effective conductance $\cCeff(a,b)$ when considering $G$ as a resistor network with edge resistances $r(e)=\si(e)^{-1}$. Finally, $\Mod_{\infty,\si}(\Ga)$ is the reciprocal of shortest-path distance between $a$ and $b$ with $\sigma$ providing the edge lengths.
\end{example}

\subsection{Spanning Tree Modulus and the MEO problem}\label{sec:spt-meo}
Although the main theorems of this paper can be developed in more generality, we make two simplifying assumptions now that help reduce notational overhead.  First, we restrict ourselves from here on to the case $p=2$.  Second, we consider only unweighted (i.e., $\si\equiv 1$) multigraphs.  Exponents other than $p$ will not be addressed in this paper.  Weighted graphs are addressed briefly in Section~\ref{sec:weighted-graphs}.

Let $G=(V,E)$ be a graph and let $\Ga=\Ga_G$ be the family of all spanning trees of $G$.
In this case, the Fulkerson dual family $\hat{\Ga}$ can be interpreted as the set of (weighted) feasible
  partitions (see \cite{chopra1989}).
\begin{definition}\label{def:feasible-part}
  A \textit{feasible partition} $P$ of a graph $G=(V,E)$ is a
  partition of the vertex set $V$ into two or more subsets,
  $\{V_1, \ldots , V_{k_P}\}$, such that each of the induced subgraphs
  $G(V_i)$ is connected. The corresponding edge set, $E_P$, is defined
  to be the set of edges in $G$ that connect vertices belonging to different $V_i$'s.
\end{definition}
Feasible partitions play an important role in characterizing homogeneous graphs in Section~\ref{sec:homog-and-fp}. The results of~\cite{chopra1989} imply that the Fulkerson dual of $\Ga_G$ is the set of all vectors
  \begin{equation}\label{eq:feasible-partition-usage}
   \frac{1}{k_P-1}\ones_{E_P},
  \end{equation}
with $P$ ranging over all feasible partitions.
Note that, by Fulkerson duality (\ref{eq:dual-dual}), the extreme points of $\Adm(\hat{\Ga})$ are (indicator functions of) spanning trees.  By convexity, any $\mu\in\cP(\Ga)$ gives rise to an admissible density $\eta=\mathcal{N}^T\mu\in\Adm(\hat{\Ga})$.  In particular, the unique optimal density $\eta^*$ for $\Mod_2(\hat{\Ga})$ belongs to the convex hull of $\Ga$.
In other words, there is an optimal probability measure $\mu^*\in\cP(\Ga)$  such that
\[
\eta^*(e)=\sum_{\ga\in\Ga}\mu^*(\ga)\cN(\ga,e)=\left(\cN^T \mu^*\right)(e)=\bE_{\mu^*}\left(\cN(\underline{\ga},e)\right)\qquad\forall e\in E.
\]
When the family of spanning trees is endowed with the standard indicator function usage, the $\cN(\underline{\ga},e)$ are indicator random variables, hence the expected edge usage becomes an edge probability:
\[
\eta(e)=\bE_{\mu}\left(\cN(\underline{\ga},e)\right)=\bP_{\mu}\left(e\in\underline{\ga} \right)\qquad\forall e\in E.
\]
Finally, since $p=2$ and $\si\equiv 1$, the energy in this case is
\begin{align}
\cE_2(\eta) & =\eta^T\eta  = \sum_{e\in E} \bP_{\mu}\left(e\in\underline{\ga} \right)^2 \label{eq:eta-energy}\\
&=\mu^T\cN\cN^T\mu  =\sum_{\ga,\ga'}\sum_{e\in E}\cN(\ga,e)\cN(\ga',e)\mu(\ga)\mu(\ga') \notag\\
& = \sum_{\ga,\ga'}|\ga\cap\ga'|\mu(\ga)\mu(\ga')=\bE_{\mu}\left(|\ga\cap\ga'|\right). \notag\\ \notag
\end{align}
This shows that $\Mod_2(\hat{\Ga})$ is equal to the minimum expected overlap in the MEO problem. Also since $\Mod_2(\hat{\Ga})$ always has  a unique optimal density $\eta^*$, this shows that the optimal edge usage probabilities are unique.  Using Lagrangian duality and the KKT conditions, the theory can be summarized as follows.
\begin{theorem}[\cite{apc:jan2016}]
  \label{thm:rho-eta-mu-kkt}
Let $G=(V,E)$ be graph, let $\Ga=\Ga_G$ be the family of spanning trees of $G$, and let $\hat{\Ga}$ be its Fulkerson dual family. Then $\rho\in \R^E_{\ge 0}$,
  $\eta\in \R^E_{\ge 0}$ and $\mu\in\cP(\Ga)$ are optimal respectively for $\Mod_{2}(\Ga)$, $\Mod_{2}(\hat{\Ga})$, and $\MEO(\Ga)$  if and only if the following conditions are satisfied.
  \begin{gather}
    \rho\in\Adm(\Gamma),\quad \eta = \N^T\mu,
    \label{eq:pdfeas-kkt}\\
    \eta(e) =
    \frac{\rho(e)}{\Mod_{2}(\Gamma)}\quad\forall
    e\in E,
    \label{eq:rho-vs-eta-kkt}\\
    \mu(\gamma)(1-\ell_\rho(\gamma)) = 0\quad\forall\gamma\in\Gamma.
    \label{eq:comp-slack-kkt}
  \end{gather}
In particular,
\begin{equation}\label{eq:meo-mod}
\MEO(\Ga)=\Mod_{2}(\hat{\Ga})=\Mod_{2}(\Ga)^{-1}.
\end{equation}
\end{theorem}
Therefore, by  (\ref{eq:ptwofulkerson}) and (\ref{eq:eta-star-rho-star-p2}), the value of the \text{MEO} problem, as well as, the optimal edge usage probabilities $\eta^*$, can be obtained by computing $\Mod_{2}(\Ga_G)$ and the optimal density $\rho^*$ instead.  An efficient algorithm for computing $\rho^*$ is described in Section~\ref{sec:algorithm}, thus leading to a means of solving the MEO problem.
\begin{remark}
To complete the picture in Theorem \ref{thm:rho-eta-mu-kkt}, one could use (\ref{eq:dual-dual}) and introduce a measure $\nu\in\cP(\hat{\Ga})$, so that $\rho^*(e)$ can be interpreted as the expected edge usage of random objects in $\hat{\Ga}$. However, we will not pursue this aspect of the theory here.
\end{remark}

The following corollary can be helpful in proving optimality.
\begin{corollary}\label{cor:weak-duality}
  Let $\Gamma$ be a finite family of objects on $G$ and let $\hat{\Gamma}$ be its
  Fulkerson dual family. Let $\rho\in\Adm(\Gamma)$ and $\eta\in\Adm(\hat{\Gamma})$.
  Then
  \begin{equation*}
    \E_{2}(\rho)\E_{2}(\eta) \ge 1.
  \end{equation*}
  Moreover, this inequality is satisfied as equality if and only if $\rho$
  and $\eta$ are optimal for their corresponding modulus problems.
\end{corollary}
\begin{proof}
  The proof follows from the definition of modulus as a minimization
  over the admissible set, which shows that
  \begin{equation*}
    \E_{2}(\rho) \ge \Mod_{2}(\Gamma) =
    \Mod_{2}(\hat{\Gamma})^{-1} \ge \E_{2}(\eta)^{-1},
  \end{equation*}
  with equality holding if and only if $\rho$ and $\eta$ solve
  their corresponding optimization problems.
\end{proof}

\subsection{An algorithm for spanning tree modulus}
\label{sec:algorithm}

As mentioned before, the family $\Ga_G$ of all spanning trees of $G$ typically contains a
very large number of objects. For instance, the complete graph $K_N$ has $N^{N-2}$ spanning trees.  In principle, since every spanning tree of $G$ gives rise to a constraint in the modulus problem, it would be computationally infeasible even to enumerate all constraints, let alone provide them as input to a standard quadratic program solver.  However,  the basic algorithm~\cite{APCDSG} described below allows one to solve the problem to within a given tolerance by iteratively ``growing'' a subfamily $\Ga'\subset \Ga_G$ with approximately the same modulus. In practice, this algorithm typically halts after discovering only a relatively small set of  ``important'' trees.

\begin{algorithm}
  \begin{enumerate}
  \item {\bf Start:} Set $\Ga^\prime=\emptyset$ and $\rho\equiv 0$.
  \item {\bf Repeat:}
    \begin{enumerate}
    \item Find $\ga\in \Ga_G\setminus\Ga^\prime$ such that
      $\ell_\rho(\ga)<1-\etol$.  Stop if none found.
    \item Add $\ga$ to $\Ga^\prime$.
    \item Optimize $\rho$ so that $\cE_p(\rho)=\Mod_p(\Ga^\prime)$.
    \end{enumerate}
  \end{enumerate}
  \caption{Basic $p$-modulus algorithm with tolerance $\etol\ge 0$.}\label{alg:mod}
\end{algorithm}

Algorithm \ref{alg:mod} is suggested by the $\Ga$-monotonicity of $p$-modulus, i.e., the fact that $\Ga_1\subset\Ga_2$ implies $\Mod_p(\Ga_1) \le\Mod_p(\Ga_2)$.  This is an example of an exterior point method; the initial guess for $\rho$ is infeasible and each step attempts to move $\rho$ closer to feasibility while guaranteeing that it is optimal for a subproblem.

The initial conditions and step (c) guarantee that, at the beginning of each pass through the loop, $\rho$ is optimal for the modulus of the subfamily $\Ga'\subset\Ga_G$.  Step (a) ensures that the algorithm terminates only when $\rho$ is approximately feasible (and approximately optimal) for $\Mod_p(\Ga_G)$.  If $\etol=0$, the algorithm will terminate the first time $\Mod_p(\Ga')=\Mod_p(\Ga_G)$, since in this case all active constraints will have been discovered.  In this case, step (c) ensures that $\rho$ is optimal for $\Mod_p(\Ga_G)$.  For $\etol>0$, it is possible to control the relative errors in the modulus value as well as in the optimal $\rho$; see \cite[Theorem~9.1]{APCDSG} for a precise statement.

Note that the algorithm described does not specify how $\ga$ should be selected in part (a).  It is known~\cite[Theorem~3.5(iv)]{apc:jan2016} that there exists a subfamily $\Ga'\subset\Ga_G$ of size $|\Ga'|\le|E|$ such that $\Mod_p(\Ga')=\Mod_p(\Ga_G)$.  Ideally, one would like to add one of the trees in this family on each step.  However, these constraints are not known a priori.  One approach that works well experimentally, is to choose the most violated (i.e., the $\gamma$ with smallest $\ell_{\rho}(\gamma)$) on each step.  This can be done efficiently for spanning tree modulus, e.g., by Kruskal's algorithm.  In practice, choosing the lightest tree on each step tends to lead to the algorithm ending with a relatively small $\Gamma'$.

Algorithm~\ref{alg:mod} can be used to solve the MEO problem as follows.  By applying standard techniques from convex optimization, one can arrange that part (c) yield, not only the density $\rho$ that is optimal for $\Mod_2(\Ga')$, but also the corresponding Lagrange dual variables $\lambda\ge 0$ enforcing the $\cN\rho\ge\one$ constraints.  As shown in~\cite[Theorem~5.1]{apc:jan2016}, when $\lambda$ is normalized to a pmf $\mu$, the resulting $\mu$ is optimal for the MEO problem.

Therefore, the algorithm not only gives an approximation for modulus, but it also provides a subfamily $\Ga'$ as well as a pmf $\mu\in\cP(\Ga')$ such that
\begin{equation*}
  \bE_{\mu}|\underline{\gamma}\cap\underline{\gamma}'| = \MEO(\Ga')\approx
  \MEO(\Ga_G).
\end{equation*}
To our knowledge, there is no other method for effectively computing the solution to the MEO problem.  The basic algorithm for computing modulus has played a pivotal role in our investigations. For instance, much of the impetus to study homogeneous cores began when trying to explain Figure~\ref{fig:spt}.

\begin{figure}
  \centering
  \includegraphics[trim=1.5cm 1.5cm 1.1cm 1.1cm,clip,width=0.6\textwidth]{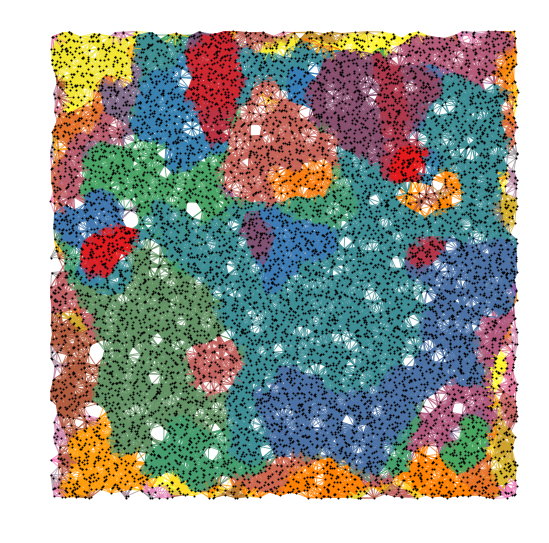}
\caption{Spanning tree modulus on a random geometric graph.  Edge colors indicate expected edge usage.  The ``islands'' of constant color are related to the deflation process described in Section~\ref{sec:deflation}.}\label{fig:spt}
\end{figure}

In the case of spanning trees,  there are ways of significantly improving the performance of Algorithm \ref{alg:mod}.  The two time-critical parts of the algorithm are the \emph{constraint selection} (Step 2(a)) and the \emph{subproblem optimization} (Step 2(c)).  For the former, we use an efficient ``smallest weighted object'' method, such as Kruskal's minimum spanning tree algorithm, while, for the latter, we implemented   a dual method based on the algorithm of Goldfarb and Idnani~\cite{goldfarb1983} that can move efficiently from the solution of one subproblem to the next. As a result, we were able to compute the spanning tree modulus of the random geometric graph shown in Figure~\ref{fig:spt} (containing 10,000 nodes and 151,280 links) with a tolerance $\etol=10^{-4}$ on a desktop CPU in under 10 minutes.

\subsection{Weighted graphs}
\label{sec:weighted-graphs}

As mentioned previously, this paper is primarily focused on unweighted multigraphs.  It turns out that this is not as restrictive as it might at first seem.  Suppose we have a weighted multigraph $G=(V,E,\si)$ with weights $\si\in\mathbb{R}^E_{>0}$.  Since modulus is continuous in $\si$~\cite[Theorem 7.2]{acfpc}, we may as well assume $\si\in\mathbb{Q}^E_{>0}$; the optimal $\rho$, $\mu$ and $\eta$ are all continuous in $\si$.  Moreover, it is straightforward to show that, for any $s>0$, modulus satisfies
\begin{equation*}
\Mod_{p,s\si}(\Ga) = s\Mod_{p,\si}(\Ga),
\end{equation*}
with exactly the same optimal density $\rho$.  Thus, by approximation, we may replace any weighted multigraph with a multigraph with positive integer weights.

Now, given a weighted multigraph $G=(V,E,\si)$ with $\si\in\mathbb{Z}^E_{>0}$, we can transform $G$ into an unweighted multigraph $G'=(V,E')$ by treating the weights $\si(e)$ as edge multiplicities.  In other words, each $e\in E$ gives rise to $\si(e)$ parallel edges in $E'$.  Keeping in mind the $\si$ weights as multiplicities, it is then straightforward to establish an equivalence between the MEO, FEU and spanning tree modulus problems on $G$ and $G'$.

\subsection{Summary of results}

The key contributions of this paper can be summarized briefly as follows.
\begin{itemize}
\item The serial rule developed in Theorem~\ref{thm:serial} provides conditions under which the MEO problem splits into two simpler problems.  This is a generalization of the well-known serial rule in resistor networks which applies when the family of paths between terminals splits into paths from the first terminal to an articulation point and paths from the articulation point to the second terminal.
\item Section~\ref{sec:homog-unif} introduces the concepts of homogeneous graphs (with respect to spanning tree modulus).  These graphs play the role of ``atoms'' in the deflation process and we provide several useful characterizations of them.
\item Section~\ref{sec:homog-and-fp} establishes the connection between homogeneous graphs and feasible partitions.  In particular, Theorem~\ref{thm:eta-max-feasible-partition} shows how the solution of the MEO problem leads directly to the solution to the \emph{minimum feasible partition} problem:
\begin{equation*}
\begin{split}
\text{minimize}\quad&\frac{|E_P|}{k_P-1}\\
\text{subject to}\quad&P\text{ is a feasible partition.}
\end{split}
\end{equation*}
This is a generalization of the minimum (global) cut problem.
\item Theorem~\ref{thm:component} shows that every graph has a homogeneous core.  The proof is constructive, using the solution of the MEO problem to produce this core.  Section~\ref{sec:deflation} uses certain fundamental properties of cores and their complements to describe a deflation process that decomposes the MEO problem recursively into subproblems on homogeneous graphs.
\item Section~\ref{sec:denseness} establishes a connection between deflation and denseness.  In particular, Theorem~\ref{thm:densest-subgraph-is-core} shows that the homogeneous core produced in the deflation process solves the following \emph{densest subgraph} problem:
\begin{equation*}
\begin{split}
\text{minimize}\quad&\frac{|E_H|}{|V_H|-1}\\
\text{subject to}\quad&H\text{ is a nontrivial, connected,} \\
&\text{vertex-induced subgraph of }G.
\end{split}
\end{equation*}
\end{itemize}

\subsection{A road map}

The remainder of this paper is organized as follows.
We begin, in Section \ref{sec:first-examples}, by describing some simple examples.
In Section~\ref{sec:serial-rule}, we develop a useful serial rule, thus named by analogy to the case of effective resistance, i.e., modulus of connecting paths.  In Section~\ref{sec:homog-unif}, we introduce the notion of homogeneous and uniform graphs for spanning tree modulus.
Then, in Section \ref{sec:deflation}, we use the serial rule to establish the main theorems of this paper, i.e., describe the process of deflation. Finally, in Section \ref{sec:applications}, we give a few theoretical applications of the deflation process.

\section{Some motivating examples}
\label{sec:first-examples}

Consider the three
graphs shown in Figure~\ref{fig:first-examples}.  Each graph
$G=(V,E)$ is defined, as usual, by its vertex set $V$ and edge set
$E$.  For each graph $G$, let $\Gamma=\Gamma_G$ be the set of spanning
trees of $G$.  Each spanning tree $\gamma\in\Gamma$ has a natural
representation in $\{0,1\}^E$ through its indicator function $\ones_{\{e\in\ga\}}$.

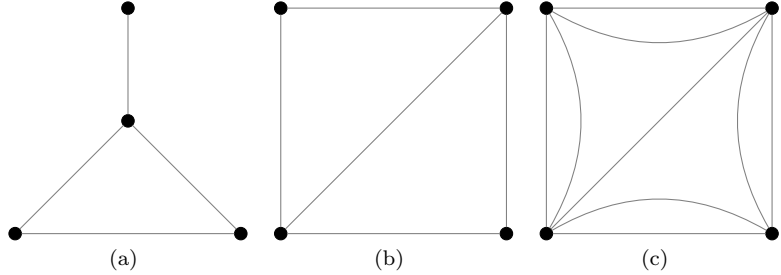
\begin{figure}
  \centering
  \subfigure[]{
    \begin{tikzpicture}[baseline,scale=1.5]
      \tikzstyle{node} = [draw,shape=circle,fill=black,scale=0.5];
      \tikzstyle{default} = [draw=black!50];
      \node[node] (a) at (0,0) {};
      \node[node] (b) at (2,0) {};
      \node[node] (c) at (1,1) {};
      \node[node] (d) at (1,2) {};
      \path[default] (a) -- (b);
      \path[default] (b) -- (c);
      \path[default] (c) -- (a);
      \path[default] (d) -- (c);
    \end{tikzpicture}
  }
  \subfigure[]{
    \begin{tikzpicture}[baseline,scale=1.5]
      \tikzstyle{node} = [draw,shape=circle,fill=black,scale=0.5];
      \tikzstyle{default} = [draw=black!50];
      \node[node] (a) at (0,0) {};
      \node[node] (b) at (2,0) {};
      \node[node] (c) at (2,2) {};
      \node[node] (d) at (0,2) {};
      \path[default] (a) -- (b);
      \path[default] (b) -- (c);
      \path[default] (c) -- (d);
      \path[default] (d) -- (a);
      \path[default] (a) -- (c);
    \end{tikzpicture}
  }
  \subfigure[]{
    \begin{tikzpicture}[baseline,scale=1.5]
      \tikzstyle{node} = [draw,shape=circle,fill=black,scale=0.5];
      \tikzstyle{default} = [draw=black!50];
      \node[node] (a) at (0,0) {};
      \node[node] (b) at (2,0) {};
      \node[node] (c) at (2,2) {};
      \node[node] (d) at (0,2) {};
      \path[default] (a) edge [bend left] (b);
      \path[default] (a) edge (b);
      \path[default] (b) edge [bend left] (c);
      \path[default] (b) edge (c);
      \path[default] (c) edge [bend left] (d);
      \path[default] (c) edge (d);
      \path[default] (d) edge [bend left] (a);
      \path[default] (d) edge (a);
      \path[default] (a) -- (c);
    \end{tikzpicture}
  }
  \caption{Graphs for the examples in
    Section~\ref{sec:first-examples}.}
  \label{fig:first-examples}
\end{figure}

\begin{example}
  Consider the graph in Figure~\ref{fig:first-examples}(a).  If the 3
  spanning trees are enumerated as
  $\Gamma=\{\gamma_1,\gamma_2,\gamma_3\}$ (the actual enumeration
  doesn't matter), then it can be seen that
  \begin{equation*}
    |\gamma_i\cap\gamma_j| =
    \begin{cases}
      3 & \text{if }i = j,\\
      2 & \text{otherwise}.
    \end{cases}
  \end{equation*}
  The symmetries in the problem allow for a straightforward
  verification that the expected overlap is minimized by the uniform
  pmf $\mu\equiv\frac{1}{3}$, yielding an expected overlap of
  \begin{equation*}
    \mathbb{E}_{\mu}|\underline{\gamma}\cap\underline{\gamma}'|
    = 3\mathbb{P}_\mu(\underline{\gamma}=\underline{\gamma}') +
    2\mathbb{P}_\mu(\underline{\gamma}\ne\underline{\gamma}')
    = 3\cdot\frac{1}{3} + 2\cdot\frac{2}{3} = \frac{7}{3}.
  \end{equation*}
\end{example}
Introducing the
indicator random variables
\begin{equation*}
  I_e :=
  \begin{cases}
    1 & \text{if }e\in\underline{\gamma}\cap\underline{\gamma}'\\
    0 &\text{otherwise}
  \end{cases}
\end{equation*}
for each $e\in E$, and using linearity of expectation along with independence of $\underline{\gamma}$ and $\underline{\gamma}'$ shows that
\begin{equation}
  \label{eq:overlap-vs-usage}
  \mathbb{E}_{\mu}|\underline{\gamma}\cap\underline{\gamma}'|
  = \sum_{e\in E}\mathbb{E}_\mu(I_e)
  = \sum_{e\in E}\mathbb{P}_\mu(e\in\underline{\gamma}\;\text{and}\;
  e\in\underline{\gamma}')
  = \sum_{e\in E}\mathbb{P}_\mu(e\in\underline{\gamma})^2.
\end{equation}
Since every spanning tree $\gamma\in\Gamma$ contains exactly $|V|-1$
edges, it follows that the average edge usage probability is
\begin{equation*}
  \frac{1}{|E|}\sum_{e\in E}\mathbb{P}_\mu(e\in\underline{\gamma})
  = \frac{1}{|E|}\sum_{e\in E}\sum_{\gamma\in\Gamma}\mathbbm{1}_\gamma(e)\mu(\gamma)
  = \frac{1}{|E|}\sum_{\gamma\in\Gamma}\mu(\gamma)\sum_{e\in E}\mathbbm{1}_\gamma(e)
  = \frac{|V|-1}{|E|},
\end{equation*}
which is a constant independent of $\mu$.  This fact together with~\eqref{eq:overlap-vs-usage} shows that, in this case, the MEO problem~\eqref{eq:min-overlap} is equivalent to minimizing the variance of the edge usage probabilities.  This fact is proved for general graphs in Theorem~\ref{thm:variance}.

\begin{example}
  Consider the graph in Figure~\ref{fig:first-examples}(b).  There are
  8 spanning trees of this graph, 4 which use the diagonal and 4 which
  do not.  If the uniform pmf $\mu\equiv\frac{1}{8}$ is chosen, then
  the diagonal will appear in half of the spanning trees while all
  other edges will appear with probability $\frac{5}{8}$.  Thus, the
  variance in edge usage probabilities is
  \begin{equation*}
    \frac{1}{5}\left[4\left(\frac{5}{8}\right)^2 +
      \left(\frac{1}{2}\right)^2\right] -
      \left\{\frac{1}{5}\left[4\left(\frac{5}{8}\right) +
        \left(\frac{1}{2}\right)\right]\right\}^2
     = \frac{1}{400},
  \end{equation*}
  and, from~\eqref{eq:overlap-vs-usage}, the expected overlap is
  \begin{equation*}
    \mathbb{E}_{\mu}|\underline{\gamma}\cap\underline{\gamma}'|
    = 4\left(\frac{5}{8}\right)^2 +
    \left(\frac{1}{2}\right)^2 = \frac{29}{16} = 1.8125.
  \end{equation*}

  On the other hand, consider the pmf
  \begin{equation}
    \label{eq:opt-pmf-slashed-square}
    \mu(\gamma) =
    \begin{cases}
      \frac{3}{20} & \text{if $\gamma$ contains the diagonal},\\
      \frac{2}{20} & \text{otherwise}.
    \end{cases}
  \end{equation}
  With this pmf, all edges are equally likely to occur in
  $\underline{\gamma}$, so this $\mu$ minimizes the variance of the
  edge usage probabilities and, hence, also the expected overlap,
  which can be computed as
  \begin{equation*}
    \mathbb{E}_{\mu}|\underline{\gamma}\cap\underline{\gamma}'|
    = 5\left(\frac{3}{5}\right)^2 = 9/5 = 1.8.
  \end{equation*}
\end{example}

\begin{example}
  Finally, consider the graph in Figure~\ref{fig:first-examples}(c).
  By identifying parallel edges, each spanning tree of this graph
  projects to a spanning tree of the graph in the previous example.
  Since each spanning tree of~\ref{fig:first-examples}(b) that
  contains the diagonal has 4 pre-images under this projection while
  all other spanning trees have 8 pre-images, the graph in
  Figure~\ref{fig:first-examples}(c) has 48 spanning trees.

  In this case, if $\mu\equiv\frac{1}{48}$ is the uniform pmf, then a
  straightforward computation shows that all edges have a
  $\frac{1}{3}$ probability of occurring in a random spanning tree
  $\underline{\gamma}\sim\mu$.  This gives the minimum expected overlap
  \begin{equation*}
    \mathbb{E}_{\mu}|\underline{\gamma}\cap\underline{\gamma}'|
    = 9\left(\frac{1}{3}\right)^2 = 1.
  \end{equation*}
\end{example}

\section{The serial rule}\label{sec:serial-rule}

In this section, we develop a generalization in the context of spanning tree modulus of the ``serial rule''
for resistor networks.  To start, let $G=(V,E)$ be a graph and
let $\Gamma\subset 2^E$ be a family of objects on $G$.
Let $E=E_1\cup E_2$ be a
partition of the edge set into two non-empty and non-overlapping subsets.  Such a
partition produces two restriction operators, for $i=1,2$,
\begin{align}\label{eq:restriction-operator}
\psi_i:  2^E &\rightarrow 2^{E_i} \\
           S &\mapsto S\cap E_i  \notag
\end{align}
and hence two induced families of objects, for $i=1,2$,
\begin{equation}
  \label{eq:Gamma-i}
  \Gamma_i := \left\{\gamma\ones_{E_i} : \gamma\in\Gamma\right\}.
\end{equation}
\begin{definition}
Given a partition $\{E_1,E_2\}$ of $E$, a pmf $\mu\in\P(\Gamma)$ naturally restricts to two new pmfs
$\mu_i\in\P(\Gamma_i)$, for $i=1,2$, called the {\it marginals} (or {\it push-forwards}),  defined as
\begin{equation}\label{eq:restrict-mu}
\mu_i(\gamma):=((\psi_i)_*\mu)(\gamma)=\sum_{\substack{T \in 2^E\\\psi_i(T)=\gamma}}\mu(T),\qquad\forall \ga\in 2^{E_i}.
\end{equation}
On the other hand, given two measures  $\nu_i: 2^{E_i}\rightarrow\R$, $i=1,2$, define the \emph{trivial coupling} (or {\it product measure}) $\nu_1\oplus \nu_2$ as
\begin{equation}\label{eq:product-measure}
 (\nu_1 \oplus \nu_2)(S):=\nu_1(\psi_1(S)) \nu_2(\psi_2(S))\qquad
 \forall S\in 2^E.
\end{equation}
\end{definition}

\begin{definition}
  We say that a partition $\{E_1, E_2\}$ for $E$  \emph{divides} $\Gamma$ if
\begin{equation}\label{eq:part-divides}
  \Gamma = \Gamma_1 \oplus \Gamma_2 :=
  \left\{
    \gamma_1 \cup \gamma_2 : \gamma_i\in\Gamma_i,\;i=1,2
  \right\}.
\end{equation}
\end{definition}
In other words, when a partition divides $\Ga$, an object is in $\Gamma$ if and only if it is the
union of an arbitrary object in $\Gamma_1$ and another in $\Gamma_2$.
In that case, the modulus/MEO problems split into
two smaller subproblems.
\begin{theorem}\label{thm:serial}
  Let $\Gamma\subset 2^E$ be a family of objects on the graph $G=(V,E)$,
  and let $E=E_1\cup E_2$ be a partition that divides $\Gamma$. Let $\Ga_i$, $i=1,2$, be the induced families as in (\ref{eq:Gamma-i}).  Then:
\begin{enumerate}
\item the following serial rule holds
\begin{equation}\label{eq:meo-serial}
\MEO(\Ga)=\MEO(\Ga_1)+\MEO(\Ga_2);
\end{equation}
 \item  a pmf $\mu\in\P(\Gamma)$ is optimal for $\MEO(\Gamma)$
  if and only if its marginal pmfs $\mu_i\in\cP(\Ga_i)$, $i=1,2$, defined in (\ref{eq:restrict-mu}),  are optimal for $\MEO(\Ga_i)$, respectively;
\item conversely, given two pmfs $\nu_i\in\cP(\Ga_i)$ that are optimal for $\MEO(\Ga_i)$, for $i=1,2$, then $\nu_1\oplus\nu_2$ is an optimal pmf in $\cP(\Ga)$ for $\MEO(\Ga)$;
\item finally, for any pmf $\mu$ with marginals $\mu_i$, if $e\in E_i$, $i=1,2$, then
\[
\bP_\mu(e\in\underline{\ga})=\bP_{\mu_i}(e\in \underline{\ga_i}).
\]
\end{enumerate}
\end{theorem}
\begin{remark}
Note that a pmf $\mu\in\mathcal{P}(\Ga)$ is not necessarily the trivial coupling of its marginals.  In particular, Theorem~\ref{thm:serial}(2) does not imply that  $\mu=\mu_1\oplus\mu_2$; rather, it shows that any coupling of $\mu_1$ and $\mu_2$ is optimal for $\MEO(\Gamma)$, including $\mu_1\oplus\mu_2$.
\end{remark}
\begin{proof}[Proof of Theorem \ref{thm:serial}]
Using  the fact that the partition divides $\Ga$, we see that for any $\ga,\ga'\in\Ga$, there are $\ga_i,\ga'_i\in\Ga_i$, $i=1,2$, so that $\ga=\ga_1\cup\ga_2$ and $\ga'=\ga'_1\cup\ga'_2$. Therefore, by the overlap formula (\ref{eq:min-overlap-gen}),
\begin{equation}\label{eq:split-overlap}
  \begin{split}
C(\ga,\ga') & =\sum_{e\in E}\cN(\ga,e)\cN(\ga',e)\\
& = \sum_{e\in E_1}\cN(\ga_1,e)\cN(\ga'_1,e) + \sum_{e\in E_2}\cN(\ga_2,e)\cN(\ga'_2,e)\\
& = \sum_{e\in E}\cN(\ga_1,e)\cN(\ga'_1,e) + \sum_{e\in E}\cN(\ga_2,e)\cN(\ga'_2,e)\\
& = C(\ga_1,\ga'_1)+C(\ga_2,\ga'_2).
\end{split}
\end{equation}

Using the fact that the partition divides $\Ga$ together with~\eqref{eq:split-overlap} yields
\begin{align*}
\bE_{\mu}[C(\underline{\ga},\underline{\ga'})] & =\sum_{\ga,\ga'\in\Ga}\mu(\ga)\mu(\ga')C(\ga,\ga')\\
& = \sum_{\ga_1,\ga'_1,\ga_2,\ga'_2}\mu(\ga_1\cup\ga_2)\mu(\ga'_1\cup\ga'_2)(C(\ga_1,\ga'_1)+C(\ga_2,\ga'_2)).
\end{align*}
The summand splits into two terms.  For the first term, observe that
\begin{equation*}
  \sum_{\ga_1,\ga'_1,\ga_2,\ga'_2}\mu(\ga_1\cup\ga_2)\mu(\ga'_1\cup\ga'_2)C(\ga_1,\ga'_1)
  = \sum_{\ga_1,\ga'_1}\mu_1(\ga_1)\mu_1(\ga'_1)C(\ga_1,\ga'_1),
\end{equation*}
where $\mu_1$ is the first marginal of $\mu$. The second term yields a similar expression with $\mu_1$ replaced by the second marginal $\mu_2$. This implies that, if $\mu\in\cP(\Ga)$, and $\mu_i$, $i=1,2$, are its marginals, then
\begin{equation}\label{eq:eo-mu-marginals}
\bE_{\mu}[C(\underline{\ga},\underline{\ga'})] = \bE_{\mu_1}[C(\underline{\ga_1},\underline{\ga'_1})] + \bE_{\mu_2}[C(\underline{\ga_2},\underline{\ga'_2})].
\end{equation}
Moreover,
\begin{equation}\label{eq:first-meo-upperbd}
  \MEO(\Ga_1)+\MEO(\Ga_2)  \le \bE_{\mu_1}[C(\underline{\ga_1},\underline{\ga'_1})] + \bE_{\mu_2}[C(\underline{\ga_2},\underline{\ga'_2})].
\end{equation}
Applying~(\ref{eq:eo-mu-marginals}) and~\eqref{eq:first-meo-upperbd} to an optimal pmf  $\mu$, we thus get that
\begin{equation}\label{eq:second-meo-upperbd}
\MEO(\Ga_1)+\MEO(\Ga_2)  \le \bE_{\mu}[C(\underline{\ga},\underline{\ga'})]=\MEO(\Ga).
\end{equation}

Conversely, suppose $\nu_i\in\P(\Gamma_i)$, $i=1,2$, are optimal for their respective $\MEO$
  problems. Let $\nu:=\nu_1\oplus\nu_2$ be defined as in (\ref{eq:product-measure}).
Then $\nu$ is supported on $\gamma$ of the form $\ga_1\cup \ga_2$, where $\ga_i$ is in the support of $\nu_i$, and hence in $\Ga_i$, for $i=1,2$. Since the partition $\{E_1,E_2\}$ divides $\Ga$, this means that the support of $\nu$ is in $\Ga$. Moreover,
\begin{align*}
\sum_{\substack{\ga_1\in \Ga_1 \\ \ga_2\in\Ga_2}}\nu(\ga_1\cup\ga_2)  =
\sum_{\substack{\ga_1\in \Ga_1 \\ \ga_2\in\Ga_2}}\nu_1(\ga_1)\nu_2(\ga_2) =\left(\sum_{\ga_1\in\Ga_1}\nu_1(\ga_1)\right)
\left(\sum_{\ga_2\in\Ga_2}\nu_2(\ga_2)\right)=1.
\end{align*}
So $\nu$ is a pmf in $\cP(\Ga)$. Also,
\[
\sum_{\ga_2\in\Ga_2}\nu(\ga_1\cup\ga_2)=\nu_1(\ga_1)\sum_{\ga_2\in\Ga_2}\nu_2(\ga_2)=\nu_1(\ga_1),
\]
so $\nu_1$ is a marginal of $\nu$, and the same can be said of $\nu_2$. In particular,
(\ref{eq:eo-mu-marginals}) and (\ref{eq:first-meo-upperbd}) hold with $\nu$ and $\nu_i$'s instead of $\mu$ and $\mu_i$'s. Finally,
\begin{align}
\MEO(\Ga) & \le \bE_{\nu}[C(\underline{\ga},\underline{\ga'})]= \bE_{\nu_1}[C(\underline{\ga_1},\underline{\ga'_1})] + \bE_{\nu_2}[C(\underline{\ga_2},\underline{\ga'_2})] \label{eq:third-meo-upperbd}\\
& = \MEO(\Ga_1)+\MEO(\Ga_2). \notag
\end{align}
Therefore, the inequalities in (\ref{eq:second-meo-upperbd}), with $\nu$'s instead of $\mu$'s, and in (\ref{eq:third-meo-upperbd}) are both equalities.
This shows that parts~1 and~3 hold.

For part 2, assume that $\mu\in\cP(\Ga)$ is optimal. Then by part 1, the inequality in (\ref{eq:second-meo-upperbd}) is an equality and, by~\eqref{eq:eo-mu-marginals},
\[
0=\left( \bE_{\mu_1}[C(\underline{\ga_1},\underline{\ga'_1})]-\MEO(\Ga_1)\right) + \left(\bE_{\mu_2}[C(\underline{\ga_2},\underline{\ga'_2})] -\MEO(\Ga_2)\right).
\]
Since each quantity in parenthesis is non-negative, they both have to be zero.
Namely, the marginals of $\mu$ are optimal as well for their respective $\MEO$ problems.

Conversely, suppose that $\mu$ is a pmf on $\Ga$ such that both of its marginals $\mu_i$ are optimal for their respective $\MEO$ problems.  Then~(\ref{eq:first-meo-upperbd}) holds as equality and thus, by~(\ref{eq:eo-mu-marginals}) and part 1,  $\mu$ is optimal for $\MEO(\Ga)$. This concludes the proof of part~2.

For part 4, assume $\mu$ is a pmf in $\cP(\Ga)$ with marginals $\mu_i$, $i=1,2$.
 If $e\in E_1$, then
  \begin{align*}
      \mathbb{P}_{\mu}(e\in\underline{\gamma}) &=
      \sum_{\gamma\in\Gamma}\cN(\gamma,e)\mu(\gamma) \\
& = \sum_{\substack{\gamma_1\in\Gamma_1\\\gamma_2\in\Gamma_2}}
      \cN(\gamma_1\cup\gamma_2,e)\mu(\gamma_1\cup\gamma_2)
&\text{(by (\ref{eq:part-divides}))} \\
&= \sum_{\gamma_1\in\Gamma_1}\cN(\gamma_1,e)
      \sum_{\gamma_2\in\Gamma_2}\mu(\gamma_1\cup\gamma_2)
&\text{(since $e\in E_1$)}\\
&  =
      \sum_{\gamma_1\in\Gamma_1}\cN(\gamma_1,e)\mu_1(\gamma_1)
&\text{(by (\ref{eq:restrict-mu}))}
 \\
&      = \mathbb{P}_{\mu_1}(e\in\underline{\gamma}_1).
  \end{align*}
A similar argument holds when $e\in E_2$.
\end{proof}

In fact, even when a partition doesn't divide the family, the serial
rule still provides a bound.
\begin{corollary}
  Let $\Gamma\subset 2^E$ be a family of objects on the graph $G=(V,E)$,
  let $E=E_1\cup E_2$ be a partition of $E$, and let $\Gamma_1$ and $\Gamma_2$ be defined as in~\eqref{eq:Gamma-i}.  Then
  \begin{equation*}
    \Mod_{2}(\Gamma)^{-1} \ge \Mod_{2}(\Gamma_1)^{-1} +
    \Mod_{2}(\Gamma_2)^{-1},
  \end{equation*}
with equality if the partition divides $\Ga$.
\end{corollary}
\begin{proof}
  This is a consequence of Theorem~\ref{thm:serial} and a monotonicity
  property of modulus.  Note that, given an arbitrary partition of $E$, it
  is always the case that $\Gamma\subseteq\Gamma_1\oplus\Gamma_2$. This
  implies that $\Adm(\Gamma)\supseteq\Adm(\Gamma_1\oplus\Gamma_2)$
  and, therefore, that
  \begin{equation*}
    \Mod_{2}(\Gamma) \le \Mod_{2}(\Gamma_1\oplus\Gamma_2).
  \end{equation*}
\end{proof}

This serial rule can be  generalized to partitions with a larger number of parts.
In the case of spanning tree modulus, this
allows us to restrict attention to biconnected graphs.  Recall that a
graph is called {\it (vertex) biconnected} if it is impossible to disconnect
the graph by removing a single vertex.  Every connected graph has a
unique decomposition into biconnected components that are connected to
each other through \emph{articulation points} (i.e., points whose
removal \emph{does} disconnect the graph).

\begin{theorem}\label{thm:serial-spt}
Let $\Ga=\Ga_G$ be the family of all spanning trees of a graph $G=(V,E)$, and let
$E=\cup_{i=1}^r E_i$ be the partition of $E$ into biconnected components. Then this partition divides $\Ga$, meaning that every spanning tree $\ga$ of $G$ is a concatenation of  trees $\ga_i$, spanning each biconnected component $E_i$. Let $\Gamma_i$ be
  defined as in~\eqref{eq:Gamma-i} for $i=1,2,\ldots,r$.  Then
  \begin{equation*}
    \Mod_{2}(\Gamma)^{-1} = \sum_{i=1}^r\Mod_{2}(\Gamma_i)^{-1}.
  \end{equation*}
\end{theorem}
\begin{proof}
  This is a consequence of Theorem~\ref{thm:serial} and the fact that
  the biconnected partition of $G$ divides the family of spanning
  trees.
\end{proof}

Essentially, in order to understand spanning tree
modulus, it is sufficient to understand it on biconnected graphs.  A
pmf $\mu\in\P(\Gamma)$ is optimal for $\MEO(\Ga)$ if and
only if its restrictions $\mu_i\in\P(\Gamma_i)$ are optimal for
the $\MEO$ problems of the associated biconnected components.

\section{Homogeneous and uniform graphs}\label{sec:homog-unif}

The examples of Section~\ref{sec:first-examples} suggest the following
definitions.

\begin{definition}\label{def:homog-unif}
  Let $G=(V,E)$ be given and let $\Ga=\Ga_G$ be the family of spanning trees of $G$. Let $\rho^*$ and $\eta^*$ be the unique extremal densities for $\Mod_{2}(\Ga)$ and $\Mod_{2}(\hat{\Ga})$ respectively.
  Then $G$ is called \emph{homogeneous} (with respect to spanning tree
  modulus) if $\eta^*$ is constant, or equivalently by (\ref{eq:rho-vs-eta-kkt}), if $\rho^*$ is constant.

 The graph $G=(V,E)$ is called \emph{uniform} (with respect
  to spanning tree modulus) if the uniform pmf $\mu_0$ defined
  as
  \begin{equation}\label{eq:mu-0}
  \mu_0(\ga) = \frac{1}{|\Ga|}
  \end{equation}
  is optimal for $\MEO(\Ga_G)$---that
  is, by (\ref{eq:pdfeas-kkt}), if $\eta^*=\N^T\mu_0$.
\end{definition}

In Figure~\ref{fig:first-examples}, for example,~(b) and~(c) are
homogeneous, while~(a) and (c) are uniform.  In what follows, we will
examine these two classes of graphs carefully and then show that every
graph can be decomposed into a sequence of homogeneous graphs through
a process called deflation.

\subsection{Homogeneous graphs}

An interesting property of spanning tree modulus that distinguishes it
from the modulus of other families on a graph (paths or cuts, for
example) is that the row sums of $\N$ are identical. Indeed, since every
spanning tree uses $|V|-1$ edges, it follows that
\begin{equation}
  \label{eq:row-sums}
  \sum_{e\in E}\N(\gamma,e) = |V|-1\quad\forall\gamma\in\Gamma.
\end{equation}
This has implications for the distribution of edge usage
probabilities. Define the {\it expectation} and {\it
  variance} of a vector $\xi\in \R^E$ as usual:
\[
\bE(\xi):= \frac{1}{|E|}\sum_{e\in E}\xi(e)
\quad\text{and}\quad \Var(\xi):= \bE(\xi^2) -
\left(\bE(\xi)\right)^2
\]
where the square in $\xi^2$ is applied element-wise.
\begin{lemma}
  \label{lem:avg-eta}
  Let $G=(V,E)$ be given, and let $\Ga_G$ be the family of
  spanning trees of $G$. Let $\mu\in\P(\Gamma_G)$ be a pmf and let
  $\eta = \N^T\mu$ be the corresponding edge usage probabilities.  Then
  \begin{equation*}
    \bE(\eta) = \frac{|V|-1}{|E|}.
  \end{equation*}
\end{lemma}

\begin{proof}
  Assuming $\mu$ and $\eta$ as stated, we have
  \begin{equation*}
    \bE(\eta) = \frac{1}{|E|}\sum_{e\in E}\eta(e)
    = \frac{1}{|E|}\sum_{e\in E}\sum_{\gamma\in\Gamma_G}\mu(\gamma)\N(\gamma,e)
    = \frac{1}{|E|}\sum_{\gamma\in\Gamma_G}\mu(\gamma)\sum_{e\in E}\N(\gamma,e)
  \end{equation*}
  and the result follows from~\eqref{eq:row-sums}.
\end{proof}

An immediate consequence of Lemma~\ref{lem:avg-eta} is the following
theorem.
\begin{theorem}
  \label{thm:variance}
  Let $G=(V,E)$ be given, with spanning trees $\Gamma_G$. Then the
  extremal density $\eta^*$ for $\Mod_{2}(\widehat{\Ga_G})$ is the unique solution to the FEU problem~\eqref{eq:min-variance}.
\end{theorem}
\begin{proof}
  Let $\eta$ be any admissible density for the Fulkerson dual family $\widehat{\Gamma_G}$. By (\ref{eq:pdfeas-kkt}), $\Mod_{2}(\widehat{\Ga_G})$ is also the value of the following minimization problem:
 \begin{equation*}
    \label{eq:prob-interpretation}
    \begin{split}
      \text{\rm minimize}\quad &
      \cE_{2}(\eta)\\
      \text{\rm subject to}\quad & \eta(e) =
      (\cN^T\mu)(e)\;
      \forall e\in E,\\
      & \mu\in\P(\Gamma_G).
    \end{split}
  \end{equation*}
Note that the energy of $\eta$ can be written as:
\begin{equation*}
      \E_{2}(\eta) = \sum_{e\in E}\eta(e)^2
      = |E|\left(\Var(\eta)-\mathbb{E}(\eta)^2\right)
\end{equation*}
Hence, by Lemma~\ref{lem:avg-eta},
\begin{equation}\label{eq:energy-variance}
 \E_{2}(\eta)
 = |E|\Var(\eta)-\frac{(|V|-1)^2}{|E|}.
\end{equation}
  establishing a straightforward equivalence between $\Mod_{2}(\widehat{\Ga_G})$ and the FEU problem.
\end{proof}
The following is a useful criterion for showing that a graph is homogeneous.
\begin{corollary}
  \label{cor:homog-density}
  Let $G=(V,E)$ be given, let $\Gamma_G$ be the family of
  spanning trees of $G$, and let $\widehat{\Gamma_G}$ be the Fulkerson dual of $\Ga_G$.  Define the density
  \begin{equation}\label{eq:eta-sigma}
    \eta_{\hom}(e) := \frac{|V|-1}{|E|}.
  \end{equation}
  Then,
  \begin{equation}\label{eq:homog-bound}
    \Mod_{2}(\widehat{\Gamma_G})\ge \E_{2}(\eta_{\hom})=\frac{(|V|-1)^2}{|E|}.
  \end{equation}
Moreover, the following are equivalent
\begin{itemize}
\item[(i)] $G$ is homogeneous,
\item[(ii)] $\eta_{\hom}\in\Adm(\widehat{\Gamma_G})$,
\item[(iii)] equality holds in (\ref{eq:homog-bound}).
\end{itemize}
\end{corollary}

\begin{proof}
  Let
  $\eta\in\Adm(\widehat{\Gamma_G})$.  Then
  \begin{equation}\label{eq:homog-var-bound}
    \Var(\eta) \ge 0 = \Var(\eta_{\hom}).
  \end{equation}
  Using~\eqref{eq:energy-variance},
\begin{equation}\label{eq:homog-energy-bound}
\cE_{2}(\eta)\ge \cE_{2}(\eta_{\hom})=\frac{(|V|-1)^2}{|E|},
\end{equation}
 and minimizing over
  $\eta\in\Adm(\widehat{\Gamma_G})$ yields~\eqref{eq:homog-bound}.

By Definition~\ref{def:homog-unif}, $G$ is homogeneous if and only if the optimal density $\eta^*$ for $\Mod_{2}(\widehat{\Ga_G})$ is constant.  By Lemma~\ref{lem:avg-eta} this is equivalent to the statement that $\eta^*=\eta_{\hom}$.  This shows that (i) implies both (ii) and (iii).  Moreover, if $\eta_{\hom}$ is admissible, then it is optimal by~\eqref{eq:homog-energy-bound}, showing that (ii) implies (i).  Finally, to see that (iii) implies (i), suppose equality holds in (\ref{eq:homog-bound}). Then $\eta=\eta^*$ attains the bounds in~\eqref{eq:homog-var-bound} and~\eqref{eq:homog-energy-bound}.  This implies that $\eta^*$ is constant and thus that $G$ is homogeneous.
\end{proof}

Combining this with the probabilistic interpretation provides another
characterization of homogeneous graphs.

\begin{corollary}
  \label{cor:homog-pmf}
  A graph $G=(V,E)$ is homogeneous if and only if there exists
  a pmf $\mu\in\P(\Gamma_G)$ such that $\bP_\mu(e\in\underline{\ga})$ does not depend on $e$.
  Moreover, if such a pmf $\mu$ exists, then it is necessarily optimal for $\MEO(\Ga_G)$, and
  \begin{equation}\label{eq:pmu-homog}
    \mathbb{P}_{\mu}\left(e\in\underline{\gamma}\right)
    = \frac{|V|-1}{|E|}.
  \end{equation}
\end{corollary}
\begin{proof}
If $G=(V,E)$ is homogeneous, then $\eta^*$ is constant.  So, by~\eqref{eq:pdfeas-kkt}, any optimal $\mu^*$ for the MEO problem satisfies $\eta^*=\cN^T\mu^*$.

Conversely, assume that there is a pmf $\mu\in\cP(\Ga_G)$ such that
$\bP_\mu(e\in\underline{\ga})$ is independent of $e$.  Then $\eta=\mathcal{N}^T\mu$ is admissible, as described in Section~\ref{sec:spt-meo}, and $\eta$ is constant.  Corollary~\ref{cor:homog-density} then shows that $G$ is homogeneous and $\eta=\eta^*$ and, thus, that $\mu$ is optimal.
Moreover, $\eta=\eta_{\hom}$, which gives~\eqref{eq:pmu-homog}.
\end{proof}

An intuitive way to interpret Theorem \ref{thm:variance} and Corollary~\ref{cor:homog-pmf} is that
homogeneity is a ``preferred property'' of a graph. A graph that fails
to be homogeneous must necessarily contain some barrier that excludes
the existence of a pmf that perfectly equalizes the edge usage
probabilities.  The next section exactly characterizes such barriers.

\subsection{Homogeneity and feasible partitions}\label{sec:homog-and-fp}

We begin by defining the weight of a feasible partition (see Definition \ref{def:feasible-part}).

\begin{definition}
	Let $P$ be a feasible partition on $G=(V,E)$.  The \emph{weight} of $P$ is defined as
    \begin{equation*}
    w(P) = \frac{|E|}{k_P-1}.
    \end{equation*}
\end{definition}

The importance of feasible partitions in studying spanning trees comes
from the observation that, if $P$ is a feasible partition for $G$ with $k_P$ parts, then any spanning tree of $G$ must use at least $k_P-1$ edges in
$E_P$.  This provides the following lemmas.
\begin{lemma}\label{lem:fp-avg-eta}
Let $G=(V,E)$ be a graph and let $\eta^*$ be optimal
  for $\Mod_{2}(\widehat{\Ga_G})$, with $\eta^*=\N^T\mu^*$ for any optimal pmf
  $\mu^*$.  Let $P$ be a feasible partition of $G$ with $k_P$ parts.  Then 
  \begin{equation*}
   \frac{1}{|E_P|}\sum_{e\in E_P}\eta^*(e) 
   \ge \frac{k_P-1}{ |E_P|} = w(P)^{-1}.
  \end{equation*}
  If every fair tree $\gamma\in\Gamma_G^f$ has the property that $|\ga\cap E_P|=k_P-1$, then the inequality holds as equality.
\end{lemma}
\begin{proof}
  In light of the observation that every spanning tree must use at
  least $k_P-1$ edges of $E_P$, we can estimate the average
  \begin{equation}\label{eq:fp-eta-avg}
    \begin{split}
      \frac{1}{|E_P|}\sum_{e\in E_P}\eta^*(e) &=
      \frac{1}{|E_P|}\sum_{e\in
        E_P}\sum_{\gamma\in\Gamma}\mu^*(\gamma)\N(\gamma,e) \\
      &=
      \frac{1}{|E_P|}\sum_{\gamma\in\Gamma}\mu^*(\gamma)\sum_{e\in
        E_P}\N(\gamma,e)\\
      &= \frac{1}{|E_P|}\sum_{\gamma\in\Gamma}\mu^*(\gamma)
      |\ga\cap E_P|
      \ge\frac{k_P-1}{|E_P|}.
    \end{split}
  \end{equation}
  Since, by definition, $\mu^*$ is supported on fair trees, equality holds if every fair tree uses exactly $k_P-1$ edges of $E_P$.
\end{proof}
\begin{lemma}\label{lem:fp-eta}
  Let $G=(V,E)$ be a graph and let $\eta^*$ be optimal
for $\Mod_{2}(\widehat{\Ga_G})$, with $\eta^*=\N^T\mu^*$ for any optimal pmf
  $\mu^*$.  Let $P$ be a feasible partition of $G$ with $k_P$ parts.  Then there exists
  an edge $e\in E_P$ such that
  \begin{equation*}
   \eta^*(e) \ge \frac{k_P-1}{ |E_P|} = w(P)^{-1}.
  \end{equation*}
\end{lemma}

\begin{proof}
  This is a consequence of~\eqref{eq:fp-eta-avg}; at least one edge $e\in E_P$ must have usage probability $\eta^*(e)$ that is greater or equal to the average.
\end{proof}

The following lemma establishes useful necessary conditions on fair spanning trees.

\begin{lemma}\label{lem:mu-star-stability}
	Let $\mu^*$ be optimal for the MEO problem and let $\eta^*=\N^T\mu^*$ be the optimal expected edge usage.  Let $\gamma^*$ be a fair spanning tree in the support of $\mu^*$, let $e'\in E\setminus\gamma^*$, and let $C$ be the cycle created by adding $e'$ to $\gamma^*$.  Then
    \begin{equation}\label{eq:eta-attains-max-on-cycle}
    	\eta^*(e') = \max_{e\in C}\eta^*(e).
    \end{equation}
\end{lemma}
\begin{proof}
	Theorem~\ref{thm:rho-eta-mu-kkt} implies that~\eqref{eq:eta-attains-max-on-cycle} is equivalent to the statement
    \begin{equation}\label{eq:rho-attains-max-on-cycle}
    	\rho^*(e') = \max_{e\in C}\rho^*(e).
    \end{equation}
    Suppose this is not true.  Then, we must have
    \begin{equation*}
    	\rho^*(e') < \rho^*(e^*)\qquad\text{where}\qquad
        e^* = \argmax_{e\in C}\rho^*(e)\in\ga^*.
    \end{equation*}
    Let $\gamma'$ be the tree obtained by swapping $e'$ with $e^*$ in $\gamma^*$.  By the complementary slackness condition~\eqref{eq:comp-slack-kkt}, we have $\ell_{\rho^*}(\ga^*)=1$.  But then
    \begin{equation*}
    \ell_{\rho^*}(\ga') = \ell_{\rho^*}(\ga^*) - \rho^*(e^*) + \rho^*(e')
    < 1,
    \end{equation*}
    contradicting the admissibility of $\rho^*$.  Hence, the lemma must hold.
\end{proof}

\begin{theorem}\label{thm:eta-max-feasible-partition}
	Let $\mu^*$ be optimal for the MEO problem and let $\eta^*=\N^T\mu^*$.  Define
    \begin{equation*}
    E^* = \{e\in E : \eta^*(e) =  K\},\qquad
    \text{where}\quad K := \max_{e\in E}\eta^*(e).
    \end{equation*}
    Then there exists a feasible partition $P^*$ such that
    \begin{enumerate}
    \item $E^*=E_{P^*}$,
    \item $|\gamma^*\cap E_{P^*}|= k_{P^*}-1\qquad\forall\gamma^*\in\supp\mu^*$,
    \item $\eta^*(e) =  K = w(P^*)^{-1}\qquad\forall e\in E_{P^*}$,
    \item $w(P^*) = \min\{w(P):\text{$P$ is a feasible partition of $G$}\}$.
    \end{enumerate}
\end{theorem}
\begin{proof}
	For part 1, observe that removing the edges $E^*$ from $E$ will split the graph $G$ into one or more connected components.  Let $\{V_1,V_2,\ldots,V_k\}$ be the partition of $V$ induced by grouping the vertices in these connected components.  By definition, the vertex-induced graphs $G(V_i)$ are all connected.
    
    Now, let $e^*\in E^*$ be an edge connecting two nodes $x$ and $y$.  In order to establish that $x$ and $y$ belong to different partition parts $V_i$, let us assume the contrary.  Without loss of generality, then, $x,y\in V_1$.  Let $\gamma^*\in\supp\mu^*$ with the property that $e^*\in\gamma^*$.  (Since $\eta^*(e^*)>0$, such a tree must exist.)  If we remove $e^*$ from $\gamma^*$, we split $\gamma^*$ into two connected components, one containing $x$ and the other containing $y$.  Since $G(V_1)$ is connected, there must exist a path between $x$ and $y$ that avoids $E^*$ and some edge $e'$ along this path that connects the two components of $\gamma^*\setminus\{e^*\}$. Adding $e'$ to $\gamma^*$ creates a cycle that includes both $e'$ and $e^*$ and, since $e'\notin E^*$ we have $\eta^*(e')<\eta^*(e^*)$ and Lemma~\ref{lem:mu-star-stability} establishes the contradiction.
    
    The preceding argument shows that whenever $e\in E^*$, $e$ must connect two vertices in distinct partition parts.  In particular, since $E^*$ is nonempty, this implies that $k>1$ so that $\{V_1,\ldots,V_k\}$ is a feasible partition, call it $P^*$, on $G$, and that $E^*\subseteq E_{P^*}$.  On the other hand, any edge $e\in E_{P^*}$ by definition crosses between distinct $V_i$ in the partition.  By the construction of the partition, this $e$ must belong to $E^*$.  This establishes the opposite inclusion, showing that, in fact, $E^*=E_{P^*}$.
    
To see part 2, let $\gamma^*\in\supp\mu^*$.  Since $P^*$ splits $G$ into $k_{P^*}$ connected components, it follows that
    \begin{equation*}
    |\gamma^*\cap E_{P^*}|\ge k_{P^*}-1.
    \end{equation*}
    Suppose that the inequality is strict.  Then there must be at least one component, $G(V_i)$, on which $\gamma^*$ restricts as a forest rather than as a single tree.  Let $x$ and $y$ be two nodes in distinct trees of this forest.  Since $G(V_i)$ is connected, there is a path entirely in $G(V_i)$ that connects these two nodes.  Somewhere along this path, there must exist an edge $e'$ that crosses between two trees in the spanning forest induced by restricting $\gamma^*$ to $G(V_i)$.  Without loss of generality, we may assume that $e'$ connects $x$ to $y$.  The path in $\gamma^*$ that connects $x$ to $y$ must exit $G(V_i)$ and, therefore, must cross an edge $e^*$ of $E_{P^*}$.  Since $e'\notin E^*$, once again we have $\eta^*(e')<\eta^*(e^*)$, and Lemma~\ref{lem:mu-star-stability} establishes a contradiction.
    
To obtain part 3, note that part 2 along with Lemma~\ref{lem:fp-avg-eta} implies that
\begin{equation*}
	 K =  \frac{1}{|E_{P^*}|}\sum_{e\in E_{P^*}}\eta^*(e)
	 =\frac{k_{P^*}-1}{|E_{P^*}|}.
\end{equation*}

For part 4, assume $P'$ is a feasible partition such that $w(P')<w(P^*)$.  Lemma~\ref{lem:fp-eta} implies that there exists an edge $e\in E_{P'}$ such that $\eta^*(e) \ge w(P')^{-1}$.  But then part~3 implies that
    \begin{equation*}
     K \ge \eta^*(e) \ge w(P')^{-1} > w(P^*)^{-1} =  K,
    \end{equation*}
    which is a contradiction.
\end{proof}

Lemma \ref{lem:fp-eta} allows a characterization of homogeneous graphs based
entirely on feasible partitions.

\begin{theorem}\label{thm:homog-feasible-part-bound}
  A graph $G=(V,E)$ is homogeneous if and only if for every
  feasible partition $P$ of $G$ with $k_P$ parts, we have
  \begin{equation}\label{eq:chopra-adm}
    w(P) = \frac{|E_P|}{k_P-1} \ge \frac{|E|}{|V|-1}.
  \end{equation}
\end{theorem}

\begin{proof}
  Let $\eta^*$ be the optimal density for $\Mod_{2}(\widehat{\Ga_G})$.  If $G$ is
  homogeneous, then, by Corollary \ref{cor:homog-density}, $\eta^*=\eta_{\hom}$, so
  \begin{equation*}
  \eta^*(e) = \frac{|V|-1}{|E|}
  \end{equation*}
  for every edge $e\in E$.  Therefore, Lemma~\ref{lem:fp-eta} implies that~\eqref{eq:chopra-adm} holds for every feasible partition $P$.

Conversely,
  suppose~\eqref{eq:chopra-adm} holds for every feasible partition $P$.  If we can show that $\eta_{\hom}\in\Adm(\widehat{\Gamma_G})$ then, by Corollary~\ref{cor:homog-density}, we will
  have that $G$ is homogeneous.
Recall from (\ref{eq:feasible-partition-usage}), that the edge usage matrix for $\widehat{\Ga_G}$ is given by
\[
\cN(P,e)= \frac{1}{k_P-1}\ones_{\{e\in E_P\}}.
\]
So, if $P$ is a
  feasible partition, then
  \begin{equation*}
    \sum_{e\in E}\cN(P,e)\eta_{\hom}(e) =
    \frac{|V|-1}{|E|}\sum_{e\in E_P}\frac{1}{k_P-1}
    = \frac{|V|-1}{|E|}\frac{|E_P|}{k_P-1} \ge 1,
  \end{equation*}
where the last inequality holds by~\eqref{eq:chopra-adm}.
So $\cN\eta_{\hom}\ge\one$ and $\eta_{\hom}\in\Adm(\widehat{\Ga_G})$.
\end{proof}

\subsection{Uniform graphs}

As defined in Definition \ref{def:homog-unif}, a graph $G=(V,E)$ is a uniform graph if the pmf $\mu_0$ defined  in~\eqref{eq:mu-0} is optimal for $\MEO(\Ga_G)$.
Uniform graphs can be understood through the connection to effective
resistance usually attributed to Kirchhoff.  Recall that the effective
resistance of an edge, $R_{\eff}(e)$, is defined as the voltage potential difference required to pass one unit of current from one endpoint of $e$ to the other, when $G$ is viewed as a resistor network
with unit conductances.
\begin{theorem}[Kirchhoff]
  \label{thm:kirchhoff}
  Let $G=(V,E)$ be a graph and let
  $\mu_0$ be defined as in~\eqref{eq:mu-0}.  Then, for every
  $e\in E$,
  \begin{equation*}
     \mathbb{P}_{\mu_0}\left(e\in\underline{\gamma}\right)
    = R_{\eff}(e).
  \end{equation*}
\end{theorem}
\begin{corollary}
  A lower bound on spanning tree modulus is given by
  \begin{equation*}
    \Mod_{2}(\Gamma_G) \ge \left(\sum_{e\in E}R_{\eff}(e)^2\right)^{-1}.
  \end{equation*}
\end{corollary}
\begin{proof}
  Let $\eta=\N^T\mu_0$.  By (\ref{eq:ptwofulkerson}) and Theorem \ref{thm:rho-eta-mu-kkt}, we
  see that
  \begin{equation*}
    \Mod_{2}(\Gamma_G) \ge \E_{2}(\eta)^{-1}
    = \left(\sum_{e\in E}\eta(e)^2\right)^{-1}
  \end{equation*}
  and the result follows from Kirchhoff's Theorem~\ref{thm:kirchhoff}.
\end{proof}
The serial rule developed in
Theorem~\ref{thm:serial-spt} shows that, without loss of generality, we
can limit our study of spanning tree modulus to biconnected graphs. Also, a consequence of the
deflation procedure of Section~\ref{sec:deflation} will be that if $G$
is uniform and biconnected, then $G$ is homogeneous.  As might be
expected, this class of graphs has very special properties.

\begin{theorem}
  A graph $G=(V,E)$ is both homogeneous and uniform if and only
  if the effective resistance $R_{\eff}(e)$ does not depend on
  $e\in E$, i.e.,
  \begin{equation*}
    R_{\eff}(e) \equiv \frac{|V|-1}{|E|}.
  \end{equation*}
\end{theorem}
\begin{proof}
  This is a consequence of Corollary~\ref{cor:homog-pmf}
  and Kirchhoff's Theorem~\ref{thm:kirchhoff}.
\end{proof}
\begin{example}\label{ex:homog-unif}\ 
\begin{itemize}
\item[(a)] Every complete graph is homogeneous and uniform, however the graph formed from $K_3$ by adding two parallel edges to one of the original edges is neither homogeneous nor uniform.
\item[(b)] Figure~\ref{fig:first-examples}(b) gives an example of a nonuniform homogeneous graph.  If edges are added so that the diagonal has multiplicity 3 while all other edges have multiplicity 2, the resulting graph is uniform.
\item[(c)] Every cycle graph $C_N$ is homogeneous and uniform. This is a direct consequence of Corollary~\ref{cor:homog-pmf}. To get a
spanning tree on $C_N$ we remove exactly one edge from the cycle. So each edge $e$
belongs to exactly $N-1$ spanning trees, those formed by removing edges other than $e$. Therefore, the uniform distribution on spanning trees makes the variance of the edge usage probabilities vanish.
\end{itemize}
\end{example}

\section{Deflation}
\label{sec:deflation}

In this section, we introduce the process of \emph{deflating} a
nonhomogeneous graph, essentially decomposing it into
homogeneous components.  We first prove that every graph
contains a homogeneous subgraph, which we call a homogeneous core.
\subsection{The homogeneous core}
\label{sec:homog-core}

It will be useful to treat edge-induced subgraphs $H$ of $G$ simply as subsets of the edge set $E$.
In particular, we will use the sentence ``let $H\subset E$ be a subgraph of $G$''.
\begin{definition}
Let $H\subset E$ be a subgraph of $G$, we say $H$ has the {\it restriction property} if every fair tree $\ga\in\Ga_G^f$ (see Definition~\ref{def:fair-forbidden}) restricts to a spanning tree of $H$. Namely, if $\psi_H$ is the restriction operator defined as in (\ref{eq:restriction-operator}) by $\psi_H(S):=S\cap H$, then
\[
\psi_H(\Ga_G^f)\subset  \Ga_H.
\]
\end{definition}

\begin{theorem}
  \label{thm:component}
  Let $G=(V_G,E_G)$ be a connected multigraph containing at least one edge. Let $\eta^*$ be the optimal density
  for $\Mod_{2}(\widehat{\Ga_G})$. Let
  $H_{\text{min}}\subset E$ be the subgraph where $\eta^*$ attains
  its minimum. Let $H$ be a connected component of $H_{\rm min}$. Then,  the following hold:
  \begin{enumerate}
  \item $H$ has at least one edge.
  \item $H$ has the restriction property.
  \item $H$ is a vertex-induced subgraph of $G$.
  \item $H$ is itself a homogeneous graph.
  \end{enumerate}
\end{theorem}

\begin{proof}
If $G$ is homogeneous, we are done.  Otherwise, define
\begin{equation}\label{eq:kappa}
\kappa:=\min_{e\in E} \eta^*(e).
\end{equation}
The subgraph
  $H_{\rm min}\subsetneq E$ where $\eta^*$ attains
  its minimum is nonempty.  Let $H$ be a maximal connected
  component of this graph, so that Property  1 holds by construction.  Then, $H$ is a connected, edge induced
  subgraph of $H_{\text{min}}$, with the property that every
  edge $e\in E\setminus H$ incident to $H$,  satisfies
\begin{equation}\label{eq:edge-incident-h}
\eta^*(e)>\kappa.
\end{equation}

Property 2 can be seen by contradiction.  Suppose $\gamma$ is a fair tree that does not restrict to a spanning tree on $H$.  By~\eqref{eq:comp-slack-kkt}, $\gamma$ must be a $\rho^*$-minimal
  spanning tree with $\ell_{\rho^*}(\gamma) = 1$. By (\ref{eq:eta-star-rho-star-p2}),
  $\eta^*$ differs from $\rho^*$ only by a positive
  multiplicative constant. Therefore,
  $\gamma$ is also $\eta^*$-minimal.

  Now, let $\gamma_H$ be the restriction of $\gamma$ to $H$.  By assumption,  $\gamma_H$ is a forest composed of more than one tree.  Let $x,y\in H$ be vertices belonging to distinct subtrees.  Since $H$ is connected, there is a path in $H$ connecting these two vertices and an edge $e'\in H\setminus\ga_H$ along this path that crosses between two distinct trees of $\ga_H$.  Without loss of generality, we may assume this edge is between $x$ and $y$.  Since $\ga$ is a tree, it contains a path connecting $x$ and $y$ and, since $x$ and $y$ are in distinct components of $\ga_H$, this path must include an edge $e''\in\ga\setminus H$.

  By assumption 
  \begin{equation}\label{eq:edge-to-swap}
  \eta^*(e'')>\eta^*(e').
  \end{equation}  Let $\ga'$ be the spanning tree of $G$ formed by replacing the edge $e''$ in $\ga$ by $e'$.  Then, by (\ref{eq:edge-to-swap}),
 \[\ell_{\eta^*}(\gamma') <
  \ell_{\eta^*}(\gamma),
\]
  contradicting the fact that $\gamma$ is a $\eta^*$-minimal spanning tree.  Thus, it must be the case that $\gamma_H$ is a tree.

  For Property 3, we want to show that $H$, which was originally defined as an
  edge-induced subgraph, is in fact vertex induced. In other words, we must verify that if $x,y\in V_H$ and if $e\in E$ connects $x$ to $y$ then $e\in H$.  Suppose, not.  Then, by (\ref{eq:edge-incident-h}),
  $\eta^*(e)>\kappa$.
Let $\mu^*$ be an optimal pmf and let $\gamma$ be any tree in its support.  By Property 2, the unique path in $\gamma$
  connecting $x$ to $y$ lies entirely in $H$, and therefore does not
  contain the edge $e$.  Thus, no tree $\ga$ in the support of $\mu^*$
  contains the edge $e$. Hence,
  \begin{equation*}
    0 \le \kappa < \eta^*(e)
    = \mathbb{P}_{\mu^*}(e\in\underline{\gamma}) = 0,
  \end{equation*}
which is a contradiction.

Finally, to prove Property 4, recall the notations of Section \ref{sec:serial-rule}. Here the partition of $E$ is given by $H$ and $E\setminus H$. Let $\mu^*$ be an optimal pmf for $G$, and let $\mu_H$ be the marginal of $\mu$ on $H$, as defined in (\ref{eq:restrict-mu}).  Fix  $e\in H$, then
\begin{align*}
\bP_{\mu_H}(e \in\underline{\gamma})
& = \sum_{\substack{\gamma \in \Gamma_{H}\\ e \in \gamma}} \mu_H(\ga)=
\sum_{\substack{\gamma \in \Gamma_{H}\\ e \in \gamma}} \sum_{\substack{T\in \Gamma_G\\
\gamma=T\cap H}}\mu^*(T) \\
& =\sum_{\gamma\in \Gamma_{H}}
\sum_{T\in \Gamma_G} \mu^*(T)\ones_{\{e \in \gamma, \gamma=T \cap H\}}\\
&=\sum_{T\in \Gamma_G} \mu^*(T) \sum_{\gamma \in \Gamma_H} \ones_{\{e\in T, \gamma=T \cap H\}}\\
&= \sum_{T \in \Gamma_G}\mu^*(T) \ones_{\{e\in T\}} = \bP_{\mu^*}(e\in \underline{T})\\
&= \eta^*(e)=\kappa.\\
\end{align*}
This means that $\bP_{\mu_H}(e\in\underline{\ga})$ does not depend on $e\in H$.
By Corollary \ref{cor:homog-pmf}, $H$ is homogeneous.
\end{proof}

\begin{definition}\label{def:homog-core}
  Let $H$ be a connected subgraph of $G$, satisfying properties 1--4 in
  Theorem~\ref{thm:component}.  Such a subgraph is called a
  \emph{homogeneous core} of $G$.
\end{definition}

\begin{example}
The house graph shows that homogeneous cores are not unique, because the house itself satisfies properties 1--4 in Theorem~\ref{thm:component}, and so does the three-edge subgraph forming the roof.
\end{example}

Theorem~\ref{thm:component} has an interesting implication when
combined with the serial rule.
 Let  $G=(V_G,E_G)$ be a connected multigraph and let $H=(V_H,E_H)$ be a  homogeneous core of $G$. Following the notations of Section \ref{sec:serial-rule}, set $E_1:=E_H$ and $E_2:=E_G\setminus E_H$.
Then, for $i=1,2$,
let $\psi_i$ be the restriction operators as defined in (\ref{eq:restriction-operator}). Also, let $\Ga_G$  be the family of all spanning
  trees of $G$ and  $\Gamma_G^f$ the family of all fair spanning trees of $G$. We will use similar notation for $H$.

\begin{theorem}\label{thm:comp-divide}
  Let $\Gamma_G$ be the family of spanning trees on
  $G=(V_G,E_G)$ and let $H=(V_H,E_H)$ be a
  vertex-induced subgraph of $G$ with at least one edge and having the restriction property. Consider the largest subset $\Ga^*$ of $\Ga_G$ for which  $\psi_1(\Ga^*)\subseteq\Ga_H$, namely,
  \begin{equation}\label{eq:gamma-star}
    \Gamma^* :=\psi_1^{-1}(\Ga_H)\cap\Ga_G= \{\gamma\in\Gamma_G : \gamma\cap H\in\Gamma_H\}.
  \end{equation} Then, the partition $E_H\cup(E_G\setminus E_H)$ divides $\Gamma^*$:
\begin{equation}\label{eq:comp-divide}
\Gamma^* =\psi_1(\Gamma^*)\oplus\psi_2(\Gamma^*)= \Gamma_H\oplus\psi_2(\Gamma^*).
\end{equation}
\end{theorem}
\begin{remark}\label{rem:divide-fair-trees}
In particular,~\eqref{eq:comp-divide} implies that  $\psi_1(\Ga^*)=\Ga_H$. When $H$ is a homogeneous core, then $H$ is vertex induced and has the properties assumed in the theorem (Properties~1--3 of Theorem~\ref{thm:component}), which means that the set of fair trees $\Ga_G^f$ satisfies $\psi_1(\Ga_G^f)\subset\Ga_H$.
In other words, $\Ga_G^f\subset\Ga^*\subset\Gamma_G$.
\end{remark}

\begin{proof}[Proof of Theorem \ref{thm:comp-divide}]
Let $\gamma\in\Gamma^*$.  By definition,
  $\gamma_H:=\gamma\cap H\in\Gamma_H$ and
  \[
\gamma' := \gamma\setminus\gamma_H = \gamma\cap E_2\in\psi_2(\Gamma^*).
\]
  Thus $\Gamma^* \subset \Gamma_H\oplus\psi_2(\Gamma^*)$.

  To see the other inclusion, suppose $\gamma_H\in\Gamma_H$ and
  $\gamma'\in\psi_2(\Gamma^*)$ and let $\gamma:=\gamma_H\cup\gamma'$.  By
  definition of $\psi_2(\Ga^*)$, there exists $\ga^*\in\Ga^*$ such that $\ga^*\cap E_2=\ga'$. Thus letting $\tilde{\ga}_H:=\psi_1(\ga^*)$, we have
$\tilde{\gamma}_H\in\Gamma_H$, and
  $\tilde{\gamma}_H\cup\gamma'=\ga^*\in\Gamma^*$.  Using properties of spanning
  trees, we see that
\[|\gamma_H|=|\tilde{\gamma}_H| = |V_H|-1\qquad\text{and}\qquad
  |\tilde{\gamma}_H\cup\gamma'|=|\ga^*|=|V_G|-1.
\]
So $|\gamma'|=|V_G|-|V_H|$.
  This implies that $|\gamma|=|V_G|-1$.  Also, since $\gamma_H$ spans
  $H$, every $v\in V_H$ is incident to an edge in $\gamma$.  Now consider a vertex $v\in V_G\setminus V_H$.  Since
  $\ga^*=\tilde{\gamma}_H\cup\gamma'$ is a spanning tree of $G$, it must
  contain an edge $e$ that meets $v$.  By assumption,  $H$ is vertex-induced, so this edge $e$ cannot belong to $E_H$ and therefore must belong to $\gamma'$.
  Thus, $\gamma$ covers $V_G$ and has exactly $|V_G|-1$ edges.  To see
  that it is a spanning tree, then, it remains to show that $\gamma$
  is connected.

  First, let $x,y\in V_H$.  Since $\gamma_H$ forms a spanning tree of
  $H$, it must contain a path connecting $x$ to $y$.  Now, suppose
  $x\in V_G\setminus V_H$ and $y\in V_H$.  Since
  $\tilde{\gamma}_H\cup\gamma'$ is a spanning tree of $G$, it must
  contain a path connecting $x$ to $y$.  Traversing this path starting
  from $x$, consider the first time the path crosses from
  $V_G\setminus V_H$ to $V_H$.  Call this edge $e$ and its two endpoints $x'\in V_G\setminus V_H$ and $y'\in V_H$.  Since $H$ is vertex
  induced, $e$ must belong to $\gamma'$.  Thus, $\gamma'$ contains a
  path from $x$ to $y'$.  Since $\gamma_H$ is connected in $H$, it
  contains a path from $y'$ to $y$ and the concatenation of these two
  paths provides a path from $x$ to $y$ in $\gamma$.  Since every
  vertex in $V_G$ is connected to every vertex in $V_H$ by a path in
  $\gamma$, it follows that $\gamma$ is connected and therefore a
  spanning tree.  Moreover, its restriction to $E_H$ is
  $\gamma_H\in\Gamma_H$, so $\gamma\in\Gamma^*$. This concludes the proof of (\ref{eq:comp-divide}).
\end{proof}

\subsection{The deflation process}\label{ssec:deflation}

Theorem~\ref{thm:comp-divide} provides a first look at the deflation
process.  Together with Theorem~\ref{thm:serial}, it shows that finding the modulus of $\Gamma_G$ is equivalent to finding
the modulus of $\Gamma_H$ and $\psi_2(\Gamma^*)$ separately.  The remainder of the deflation process involves recognizing that the latter of these two is actually a spanning tree modulus problem as well.

Once this connection is established, it is possible to repeatedly
``deflate'' homogeneous cores of a graph until we arrive at a trivial,
single-vertex graph.  Through the deflation process, we are able to
better understand the set of optimal pmfs $\mu^*$ on the family of
spanning trees $\Gamma_G$.

To see how the deflation process works, let $G$ be a nonhomogeneous
graph, and let $H$ be a homogeneous core.  The \emph{shrunk
  graph} $G/H$ is defined as follows.  (It may be helpful at this
point to refer back to Figure~\ref{fig:deflate}.)  The vertices
$V_{G/H}$ are obtained by identifying all vertices of $V_H$ in $V_G$
as a single vertex:
\begin{equation*}
  V_{G/H} = (V_G\setminus V_H) \cup \{v_H\},
\end{equation*}
where all nodes of $V_H$ have been shrunk to a single node $v_H$.  The
shrunk graph $G/H$ can then be defined through the graph homomorphism
\begin{equation}
  \phi(x) :=
  \begin{cases}
    x & \text{if }x\in V_G\setminus V_H,\\
    v_H &\text{if }x\in V_H.
  \end{cases}\label{eq:homomorphism-phi}
\end{equation}
Edges in $G$ that lie outside of $H$ are retained in
$G/H$, edges inside of $H$ are dropped, and every edge that connected a vertex in $x\in V_G\setminus V_H$ to a vertex in $V_H$ gives rise to an edge between $x$ and $v_H$.

It is also convenient to view the homomorphism as a bijection on the edges,
\begin{equation}
  \label{eq:phi-star}
  \varphi:(E_G\setminus E_H)\to E_{G/H}.
\end{equation}

\begin{lemma}\label{lem:phi-bijection}
Using the notation of Theorem~\ref{thm:comp-divide}, let $\Ga_{G/H}$ be the family of spanning trees on $G/H$.  Define $\varphi$ as in~\eqref{eq:phi-star}.  Then
\begin{equation*}
\varphi(\psi_2(\Ga^*)) = \Ga_{G/H}.
\end{equation*}
That is, $\varphi$ is a bijection between $\psi_2(\Ga^*)$ and $\Ga_{G/H}$.
\end{lemma}

\begin{proof}
Let $S\in\psi_2(\Ga^*)$.  By definition, there exists $\ga\in\Ga_G$ such that $S=\ga\cap E_2$, so $|\ga|=|V_G|-1$.  Since $\ga_H:=\ga\cap H\in\Ga_H$, $|\ga_H|=|V_H|-1$.  Thus
\begin{equation}
|S| = |\ga| - |\ga_H| = |V_G|-|V_H| = |V_{G/H}|-1.
\end{equation}
Since $\varphi$ is a bijection, $\varphi(S)$ has the correct number of edges to be a spanning tree.

To see that $\varphi(S)$ spans $V_{G/H}$, first let $x\in V_{G/H}\setminus\{v_H\}$. Then $x=\phi^{-1}(x)\in V_G\setminus V_H$.  Since $\ga$ spans $G$, $x$ must be incident to an edge $e\in\ga$.  Since $x\notin V_H$ it follows that $e\in S$.  So the edge $\varphi(e)\in\varphi(S)$ is incident to $x$ in $G/H$.  The one remaining vertex to check is $v_H$.  Since $\ga$ is connected and spans $G$, there must be at least one edge $e\in\ga$ connecting a vertex $x\in V_G\setminus V_H$ to a vertex $y\in V_H$.  Again,$e$ cannot belong to $\ga_H$ and therefore must belong to $S$.  Thus, the edge $\varphi(e)\in\varphi(S)$ is incident to the vertex $\phi(y)=v_H$ in $G/H$.

Finally, to see that $\varphi(S)$ is connected, consider any $x\in V_{G/H}\setminus\{v_H\}$, so that $\phi^{-1}(x)=x$ and let $y\in V_H$ be arbitrary.  Since $\ga$ is connected, it contains a path connecting $x$ to $y$ and, without loss of generality, we may assume that $y$ is the only vertex in this path that is contained in $V_H$. Thus, the path lies entirely in $S$.  Under the bijection $\varphi$, this path connects $\phi(x)=x$ to $\phi(y)=v_H$.  Since $\phi(S)$ contains a path between $v_H$ and every other node of $V_{G/H}$, it is connected and therefore $\phi(S)\in\Ga_{G/H}$, which shows that $\varphi(\psi_2(\Ga^*))\subseteq\Ga_{G/H}$.

To see the other inclusion, let $\ga'\in\Ga_{G/H}$ and let $S=\varphi^{-1}(\ga')$.  Then $S\in E_G\setminus E_H$ and $|S|=|\ga'|=|V_{G/H}| = |V_G|-|V_H|$.  Let $\ga_H$ be an arbitrary spanning tree of $\ga_H$ and define $\ga=S\cup\ga_H$.  Since $\ga'=\psi_2(\ga)$, we wish to show that $\ga\in\Ga^*$.  And, since $\psi_1(\ga)=\ga_H\in\Ga_H$, it is enough to show that $\ga\in\Ga_G$.

Since $|\ga|=|S|+|\ga_H|=|V_G|-1$, $\ga$ has the correct number of edges to be a spanning tree.  Since $\ga'$ spans $V_{G/H}$, $S$ spans $V_G\setminus V_H$.  Also, as a spanning tree, $\ga_H$ spans $V_H$.  Thus the union $\ga=S\cup\ga_H$ spans $V_G$.  To see that $\ga$ is connected, first note that, since $\ga_H$ is connected in $H$, $\ga$ connects all vertices in $V_H$.  Next, let $x\in V_G\setminus V_H$.  As a spanning tree, $\ga'$ contains a path connecting $\phi(x)$ to $v_H$ in $G/H$.  Under the inverse bijection $\varphi^{-1}$, this path connects $x$ to some $y\in V_H$ in $G$.  Since $\ga$ connects every vertex in $V_G\setminus V_H$ to a vertex in $V_H$ and since $\ga$ connects all vertices of $V_H$ to each other, $\ga$ is connected in $G$.
\end{proof}

The bijection established in Lemma~\ref{lem:phi-bijection} shows that $\psi_2(\Ga^*)$ can be identified with $\Ga_{G/H}$. Since the role of $\Ga$ in a family of objects is essentially as an index set for the usage matrix, every statement about the family $\psi_2(\Ga^*)$ has an equivalent for $\Ga_{G/H}$ and vice versa.

For example, a pmf $\nu\in\cP(\psi_2(\Ga^*))$ can be pushed forward to a pmf $\varphi_*\nu\in\cP(\Ga_{G/H})$ and, similarly, a pmf $\nu'\in\cP(\Ga_{G/H})$ can be pulled back to a pmf $\phi^{-1}_*\nu'\in\cP(\psi_2(\Ga^*))$.  Using this bijection, each pmf $\mu\in\cP(\Ga^*)$ can be associated with two marginal pmfs $\mu_H\in\cP(\Ga_H)$ and $\mu_{G/H}\in\cP(\Ga_{G/H})$.  Namely,
\begin{equation}\label{eq:marginals-spt}
\begin{split}
\mu_H(\ga_H) &= \mu\left(\left\{\ga\in\Ga^*:\ga\cap H=\ga_H\right\}\right),\\
\mu_{G/H}(\ga_{G/H})
&= \mu\left(\left\{\ga\in\Ga^*:\ga\cap (E_G\setminus E_H)=\varphi^{-1}(\ga_{G/H})\right\}\right)
\end{split}
\end{equation}

With this identification, the serial rule in Theorem~\ref{thm:serial} has an interesting interpretation.

\begin{theorem}\label{thm:serial-shrunk}
Let $\Ga_G$ be the family of spanning trees on a nonhomogeneous graph $G=(V,E)$ and let $H$ be a homogeneous core.  Let $G/H$ be the corresponding shrunk graph, and let $\Ga_H$ and $\Ga_{G/H}$ be the families of spanning trees on $H$ and $G/H$ respectively.  Finally, let $\varphi$ be the bijection between $\psi_2(\Ga^*)$ and $\Ga_{G/H}$ as in Lemma~\ref{lem:phi-bijection}. Then:
\begin{enumerate}
\item the following serial rule holds
\begin{equation}\label{eq:meo-serial-spt}
\MEO(\Ga_G)=\MEO(\Ga_H)+\MEO(\Ga_{G/H}),
\end{equation}
 \item  a pmf $\mu\in\P(\Gamma_G)$ is optimal for $\MEO(\Gamma_G)$
  if and only if $\mu\in\P(\Ga^*)$ and its marginal pmfs $\mu_H$ and $\mu_{G/H}$, defined in (\ref{eq:marginals-spt}),  are optimal for $\MEO(\Ga_H)$ and $\MEO(\Ga_{G/H})$ respectively;
\item conversely, given $\mu_H\in\P(\Ga_H)$ and $\mu_{G/H}\in\P(\Ga_{G/H})$ that are optimal for their respective MEO problems,  $\mu_H\oplus(\varphi^{-1}_*\mu_{G/H})$ is an optimal pmf in $\cP(\Ga_G)$ for $\MEO(\Ga_G)$;
\item finally, for any pmf $\mu\in\P(\Ga^*)$ with marginals $\mu_H$ and $\mu_{G/H}$,
\begin{equation*}
\bP_\mu(e\in\underline{\ga})=
\begin{cases}
\bP_{\mu_H}(e\in\underline{\ga}_H)&\text{if }e\in H,\\
\bP_{\mu_{G/H}}(e\in\underline{\ga}_{G/H})&\text{if }e\in E_G\setminus E_H.\\
\end{cases}
\end{equation*}
\end{enumerate}
\end{theorem}

\begin{proof}
This is simply a restatement of Theorem~\ref{thm:serial}, making use of the bijection $\varphi$ from Lemma~\ref{lem:phi-bijection}.
\end{proof}

As described earlier, Theorem~\ref{thm:serial-shrunk} can be iterated in a process we call deflation.  Starting from a graph $G$, we know there exists a homogeneous core $H\subset G$.  This divides the MEO problem on $G$ into two subproblems: an MEO problem on $H$ and one of $G/H$.  We know that any optimal pmf $\mu\in\cP(\Ga_G)$ for $\MEO(\Ga_G)$ must necessarily lie in the more restrictive set $\cP(\Ga^*)$ as defined in~\eqref{eq:gamma-star}
Moreover, the fact that $H$ is a homogeneous core implies that
\begin{equation*}
	\MEO(\Ga_H) = \frac{(|V_H|-1)^2}{|E_H|},
\end{equation*}  
so
\begin{equation*}
	\MEO(\Ga_G) = \MEO(\Ga_{G/H}) + \frac{(|V_H|-1)^2}{|E_H|}.
\end{equation*}
Iterating this procedure on $\Ga_{G/H}$ gives a natural decomposition of any graph into a sequence of shrunk graphs, eventually terminating when a final homogeneous graph is shrunk to a single vertex.  By tracing back through the deflation process, we can completely characterize the optimal pmfs for $\MEO(\Ga_G)$ in terms of their restrictions to the various cores shrunk during the deflation process.

\subsection{Deflation and denseness}\label{sec:denseness}

Another way of looking at the deflation process is through the concept of denseness.  For a given graph $G=(V,E)$, define the {\it denseness} of $G$ as
\begin{equation*}
\theta(G) := \frac{|E|}{|V|-1}.
\end{equation*}
In other words, $\theta(G)$ is the weight of the trivial partition that separates each node into its own part.  The function $\theta$ gives a sense of how dense or sparse $G$ is.  For connected $G$, $\theta$ is minimized at $\theta(G)=1$ if and only if $G$ is a tree.  A straightforward calculation shows that, on large graphs, $\theta(G)$ is closely related to average degree, since
\begin{equation}\label{eq:theta-vs-avg-deg}
2\left(1-\frac{1}{|V|}\right)\theta(G) = \frac{1}{|V|}\sum_{x\in V}\deg(x).
\end{equation}
The next theorem shows that $\theta$ is a decreasing function on the sequence of shrunk graphs produced by deflation.

\begin{theorem}\label{thm:denseness-comparison}
Let $G=(V,E)$ be a nonhomogeneous graph and let $H$ be a homogeneous core as constructed in Theorem~\ref{thm:component}.  Then
\begin{equation*}
\theta(G/H) < \theta(G) < \theta(H).
\end{equation*}
\end{theorem}
\begin{proof}
By Theorem~\ref{thm:component} and Corollary~\ref{cor:homog-density},
\begin{equation*}
\theta(H) = \frac{|E_H|}{|V_H|-1} = \kappa^{-1} = \left(\min_{e\in E}\eta^*(e)\right)^{-1},
\end{equation*}
where $\eta^*$ is optimal for the FEU problem.

Now, let $P$ be the feasible partition of $G$ defined by grouping all vertices in $H$ into a single partition part and separating all other vertices into individual parts.  Then $k_P = |V_G|-|V_H|+1 = |V_{G/H}|$ and $|E_P| = |E_{G/H}|$ so $w(P) = \theta(G/H)$.  Furthermore, since $H$ has the restriction property, every fair tree $\gamma\in\Ga_G^f$ satisfies
\begin{equation*}
|\ga\cap E_P| = |\ga|-|\ga\cap E_H| = |V_G|-|V_H| = k_P-1.
\end{equation*}
By Lemma~\ref{lem:fp-avg-eta}, then, the average of $\eta^*$ over $E_G\setminus E_H$ is $\theta(G/H)^{-1}$.  Since, $G$ is not homogeneous, $\eta^*$ is not constant.  Therefore, there must be an edge in $E_G\setminus E_H$ on which $\eta^*(e)>\kappa = \theta(H)^{-1}$ and, thus, $\theta(G/H)<\theta(H)$.

Finally,
\begin{equation*}
\begin{split}
\theta(G) &= \frac{|E_G|}{|V_G|-1} 
= \frac{|E_{G/H}|+|E_H|}{(|V_{G/H}|-1)+(|V_H|-1)}\\
&= \frac{|V_{G/H}|-1}{(|V_{G/H}|-1)+(|V_H|-1)}\theta(G/H)
+ \frac{|V_{H}|-1}{(|V_{G/H}|-1)+(|V_H|-1)}\theta(H).
\end{split}
\end{equation*}
Since $\theta(G)$ is a nontrivial convex combination of $\theta(G/H)$ and $\theta(H)$, the theorem follows.
\end{proof}

In fact, the denseness measure $\theta$ gives a way of identifying homogeneous cores.

\begin{theorem}\label{thm:densest-subgraph-is-core}
Let $G=(V,E)$ be a graph with at least one edge.  Let $\mathcal{H}$ be the set of all vertex-induced, connected subgraphs of $G$ that contain at least one edge, and suppose $H\in\mathcal{H}$ satisfies
\begin{equation*}
\theta(H) = \max_{H'\in\mathcal{H}}\theta(H').
\end{equation*}
Then $H$ is a homogeneous core.
\end{theorem}
\begin{proof}
We need to check the four properties of Definition~\ref{def:homog-core}.  $H$ contains at least one edge and is vertex induced by definition.  To see that $H$ is homogeneous, suppose not.  Then Theorem~\ref{thm:component} implies that $H$ contains a homogeneous core $H'\subsetneq H$.  By definition $H'\in\mathcal{H}$ and Theorem~\ref{thm:denseness-comparison} implies that $\theta(H)<\theta(H')$, contradicting the maximality assumption on $\theta(H)$.  All that remains, then, is to check the restriction property.

Assume $H$ does not have the restriction property.  Then there exists an optimal pmf $\mu^*\in\mathcal{P}(\Gamma_G)$ and a tree $\ga^*$ in the support of $\mu^*$ such that $\ga^*\cap H$ is a forest.  Estimating as in~\eqref{eq:fp-eta-avg}, but with the inequality in the opposite direction, we find that
\begin{equation*}
\frac{1}{|E_H|}\sum_{e\in E_H}\eta^*(e) 
= \frac{1}{|E_H|}\sum_{\gamma\in\Gamma}\mu^*(\gamma)
      |\ga\cap E_H|
< \frac{|V_H|-1}{|E_H|}.
\end{equation*}
The strict inequality holds because $|\ga^*\cap H|<|V_H|-1$.  But this implies that $\kappa$ defined in~\eqref{eq:kappa} satisfies
\begin{equation}\label{eq:kappa-upper-bound}
\kappa < \frac{|V_H|-1}{|E_H|}.
\end{equation}

Let $H'$ be a homogeneous core of $G$ constructed as in Theorem~\ref{thm:component}. By Corollary~\ref{cor:homog-density}, 
\begin{equation}\label{eq:kappa-equality}
\kappa = \frac{|V_{H'}|-1}{|E_{H'}|}.
\end{equation}
But~\eqref{eq:kappa-upper-bound} and~\eqref{eq:kappa-equality} combined imply that
\begin{equation*}
\theta(H') > \theta(H).
\end{equation*}
Since $H'\in\mathcal{H}$, this contradicts the maximality assumption.  Thus, $H$ must have the restriction property and must, therefore, be a homogeneous core of $G$.
\end{proof}

\begin{corollary}
If $H$ is the homogeneous core constructed in Theorem~\ref{thm:component}, then $H$ is a solution to the densest subgraph problem.
\end{corollary}
\begin{proof}
Let $H$ be as assumed in the theorem and let $H'$ be any solution to the densest subgraph problem.  By Theorem~\ref{thm:densest-subgraph-is-core}, $H'$ is also a homogeneous core.  Let $\mu^*\in\cP(\Ga_G)$ be optimal for the MEO problem, let $\eta^*=\cN^T\mu^*$, and let $e\in E_H$ and $e'\in E_{H'}$.  Theorem~\ref{thm:serial-shrunk} together with Corollary~\ref{cor:homog-pmf} shows that
\begin{equation*}
\theta(H)^{-1} = \frac{|V_H|-1}{|E_H|} = \eta^*(e) = \kappa \le \eta^*(e') = \frac{|V_{H'}|-1}{|E_{H'}|} = \theta(H')^{-1},
\end{equation*}
so $\theta(H)\ge\theta(H')$.  Since $H'$ was assumed to solve the densest subgraph problem, $H$ must as well.  (That is, $\theta(H)=\theta(H')$.)
\end{proof}

Essentially, deflation can be thought of as finding $H\in\mathcal{H}$ with maximum $\theta$, shrinking it away, and then repeating the process on the resulting shrunk graph $G/H$.

\subsection{Deflation and feasible partitions}
Deflation is very closely related to the results of Section~\ref{sec:homog-and-fp}.  Essentially, feasible partitions give us a reversed view of the deflation process, which proceeds as follows.  Given a graph, $G$, we identify the minimum feasible partition $P^*$ as in Theorem~\ref{thm:eta-max-feasible-partition}.  Using arguments similar to those in the rest of that section, it is possible to show that each of the $k_{P^*}$ parts of the partition have the restriction property.

Furthermore, if each of these partition parts is shrunk to a single vertex (and self-loops are discarded), the remaining graph is the homogeneous graph produced at the end of the deflation process.  An optimal pmf $\mu$ must produce optimal marginals for the MEO problems on each partition part and for the MEO problem on the shrunk graph (through the associated homomorphism).  By this restriction property, this procedure can be repeated individually on each part of the partition, thus giving an outside-in view of the deflation process.

\section{Some applications of  the deflation process}\label{sec:applications}

To give a sense of how deflation can be used in practice, we now demonstrate its use in two applications.

\subsection{Homogeneity and regular graphs}

We now consider the notion of homogeneity for spanning tree modulus on $d$-regular graphs.
Recall that a graph is called {\it $d$-regular} if the degree of every
vertex is constant equal to some integer $d\ge 2$.  A simple, connected $d$-regular graph $G$ has
exactly $d|V|/2$ edges, implying that $d|V|$ is even.  We shall
refer to the set of simple $d$-regular, connected graphs on $N$ vertices
as $\G_{d,N}$.

To see why regular graphs are of interest, observe that the minimum node degree of a homogeneous graph cannot be too small. Indeed,
in light of Theorem~\ref{thm:serial-spt}, we may restrict our attention to biconnected graphs. Thus, removing an arbitrary vertex $x$ does not disconnect $G$, meaning that $P=\{\{x\},V\setminus\{x\}\}$ is a feasible partition. Hence,  by (\ref{eq:chopra-adm}) and~\eqref{eq:theta-vs-avg-deg},
\[
 \deg(x)=\frac{|E_P|}{k_P-1} \ge \frac{|E|}{|V|-1} = \theta(G) > \frac{1}{2|V|}\sum_{y\in V}\deg(y),
\]
so the minimum node degree in a biconnected homogeneous graph must be larger than half the average degree.
This suggests
that $d$-regular graphs might be good candidates for homogeneous graphs. Indeed, we
shall see that almost every $d$-regular graph is homogeneous.

In order to make this statement precise, we recall that the statement
``Almost every $d$-regular graph has property $X$'' means that, as
$N$ tends to infinity and $d$ remains fixed, the fraction of graphs in $\G_{d,N}$ with
property $X$ tends to $1$.  Probabilistically, as $N$ tends to
infinity, the probability that a $\G_{d,N}$ graph selected with
uniform probability has property $X$ tends to $1$.  As an example of
an ``almost every'' statement, we recall the following theorem due to
Bollob\'as~\cite[Theorem 7.32]{bollobas1985}.
\begin{theorem}[Bollob\'as]\label{thm:bolobas}
  Almost every $d$-regular graph is $d$-connected (in the vertex sense).
\end{theorem}
Similarly, we can formulate the following statement.
\begin{theorem}\label{thm:regular-homog}
  Almost every $d$-regular graph is homogeneous.
\end{theorem}
Theorem \ref{thm:regular-homog} follows
from Theorem \ref{thm:bolobas} and the following theorem.
\begin{theorem}\label{thm:chipmunks}
  If $G$ is a simple graph that is $d$-regular and $d$-connected, then $G$ is homogeneous.
\end{theorem}

\begin{proof}
A $2$-regular, $2$-connected graph is a cycle. See Example~\ref{ex:homog-unif} for a proof that all cycles are homogeneous.
Therefore, we can assume that $d\geq3$. We proceed by contradiction. Suppose $G=(V_G, E_G)$ is $d$-regular and $d$-connected, but not homogeneous. By Theorem~\ref{thm:component} there exists a homogeneous core $H\subsetneq G$.  Moreover, if $\eta^*$ is the optimal density
for $\Mod_{2}(\widehat{\Ga_G})$, then 
\begin{equation*}
\eta^*(e)=\kappa=\min_{e'\in E}\eta^*(e')\qquad\forall e\in E_H.
\end{equation*}
Since $H$ is itself a homogeneous graph, Corollary~\ref{cor:homog-pmf} implies that the extremal density $\eta_H^*$ for $\Mod_2(\widehat{\Ga_H})$ is constant and equal to $\frac{|V_H|-1}{|E_H|}$.
Finally, by part 2 of Theorem~\ref{thm:serial-shrunk},
\begin{equation}\label{eq:kappa-etah}
\eta_H^*(e)=\eta^*(e)=\kappa=\frac{|V_H|-1}{|E_H|}\qquad \forall e\in E_H.
\end{equation}

Next, observe that Lemma~\ref{lem:avg-eta} together with~\eqref{eq:theta-vs-avg-deg} implies that
\begin{equation}\label{eq:kappa-bound}
\kappa=\min_{e\in E}\eta^*(e)\le\mathbb{E}(\eta^*)=\frac{|V_G|-1}{|E_G|} = \theta(G)^{-1} < \frac{2}{d}.
\end{equation}

Now define $B\subset V_H$ as the set of nodes in $H$ that are adjacent to some node in $V_G\setminus V_H$.  That is, $B$ is the set of ``boundary'' nodes for $H$ in $G$.  Note that $B\neq\emptyset$, since $G$ is connected and $H\subsetneq G$.
We now show that $B$ cannot be too small. Namely, we claim that
\begin{equation}\label{eq:boundvertices}
|B|\geq d.
\end{equation}

There are two cases to consider. First, assume that $|V_{H}\setminus B|>0$.
In order for $H$ (and therefore $G$) to be $d$-connected, there have to be at least $d$ nodes in $B$. Otherwise, $G$ could be disconnected by removing $B$, contradicting $d$-connectedness.

Now assume that $|V_{H}\setminus B|=0$. By Theorem~\ref{thm:component}, $|E_H|>0$, and $H$ is connected. Suppose $|B|=|V_H|<d$. Then, by (\ref{eq:kappa-etah}) and (\ref{eq:kappa-bound}), and the fact that $H$ is a simple graph,
\begin{equation*}
\frac{2}{d}>\kappa  =\frac{|V_H|-1}{|E_H|}
 \ge \frac{|V_H|-1}{\left(\genfrac{}{}{0pt}{}{|V_H|}{2}\right)}=\frac{2}{|V_H|}>\frac{2}{d},
\end{equation*}
which is a contradiction. Therefore, $|B|\ge d$ in this case as well and (\ref{eq:boundvertices}) is proved.

Now we estimate the number of edges in $H$ using the fact that the degree (in $H$) of nodes in $V_H\setminus B$ is equal to $d$, while for nodes in $B$ it is at most equal to $d-1$, since the latter are adjacent to at least one node in $V_G\setminus V_H$.  So,
\begin{align*}
    2|E_H| & \leq  (d-1)|B|+d|V_{H}\setminus B|\\
     & =  (d-1)|B|+d(|V_H|-|B|) \\
     & =  d|V_H|-|B| \\
     & \le d|V_H|-d &\text{(by (\ref{eq:boundvertices}))}\\
     & = d(|V_H|-1) \\
\end{align*}
Thus, by (\ref{eq:kappa-etah}) again
\[
\kappa=\frac{|V_H|-1}{|E_H|}\geq\frac{2}{d},
\]
which contradicts (\ref{eq:kappa-bound}). So $G$ must be homogeneous.
\end{proof}
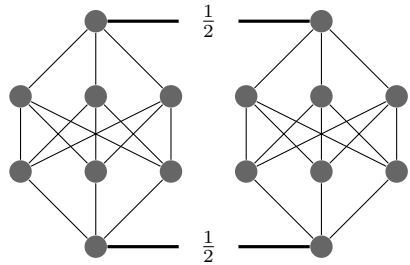
\begin{figure}[H]
\centering
\begin{tikzpicture}[main_node/.style={circle,fill=black!60,minimum size=0.5em,inner sep=3pt]}]

    \node[main_node] (1) at (0, 0) {};
    \node[main_node] (2) at (-1,1) {};
    \node[main_node] (3) at (0,1) {};
    \node[main_node] (4) at (1,1) {};
    \node[main_node] (5) at (-1,2) {};
    \node[main_node] (6) at (0,2) {};
    \node[main_node] (7) at (1,2) {};
    \node[main_node] (8) at (0,3) {};

    \node[main_node] (9) at (3,0) {};
    \node[main_node] (10) at (2,1) {};
    \node[main_node] (11) at (3,1) {};
    \node[main_node] (12) at (4,1) {};
    \node[main_node] (13) at (2,2) {};
    \node[main_node] (14) at (3,2) {};
    \node[main_node] (15) at (4,2) {};
    \node[main_node] (16) at (3,3) {};

    \draw (8) -- (6) -- (3) -- (7) -- (4) -- (6) -- (2) -- (5) -- (3) -- (1);
    \draw (1) -- (2) -- (7) -- (8) -- (5) -- (4) -- (1);

    \draw (16) -- (14) -- (11) -- (15) -- (12) -- (14) -- (10) -- (13) -- (11) -- (9);
    \draw (9) -- (10) -- (15) -- (16) -- (13) -- (12) -- (9);

\begin{scope}[>={[black]},
              every node/.style={fill=white,circle},
              every edge/.style={draw=black,very thick}]
    \path [-] (1) edge node {$\frac{1}{2}$} (9);
    \path [-] (8) edge node {$\frac{1}{2}$} (16);
\end{scope}

\end{tikzpicture}
\caption{A bi-connected, 4-regular graph. $\eta^*=0.5$ on the two labeled edges.  The unlabeled edges have $\eta^*\approx 0.467$.}\label{fig:reg-nonhomog}
\end{figure}

On the other hand, there are arbitrarily large $d$-regular graphs that
are not homogeneous.  Figure~\ref{fig:reg-nonhomog} shows a biconnected
example that can be generalized to arbitrarily large $d$ and $N$.
The graph contains two copies of the complete bipartite graph $K_{3,3}$, and can be generalized by replacing the two copies of $K_{3,3}$ with two copies of $K_{n,n}$ with $n\geq 3$. The result is an $(n+1)$-regular and biconnected graph, which is not homogeneous, because the two edges connecting the larger pieces will be used more frequently than other edges.

Moreover, it does not seem to be the case that a $d$-regular graph is
homogeneous only if it is $d$-connected.  For instance, if the $K_{3,3}$ subgraphs in Figure~\ref{fig:reg-nonhomog} are replaced by $K_{2,2}$, then the resulting graph is $3$-regular and homogeneous, but only $2$-connected.

\subsection{Uniform graphs revisited}

Here we use Theorem~\ref{thm:component}, to prove our earlier assertion that biconnected uniform graphs are homogeneous.

\begin{theorem}
  \label{thm:bicon-uniform}
Let $G=(V,E)$ be a graph.  If $G$ is a biconnected and uniform, then $G$ is homogeneous.
\end{theorem}

\begin{proof}
  We proceed by contradiction.  Suppose that $G$ is a biconnected, nonhomogeneous, uniform graph.  Theorem~\ref{thm:component} then guarantees the
  existence of a connected homogeneous component $H$.  Let $e\in E_G$ be an edge that connects a node $x\in V_H$ to a node $y\in V_G\setminus V_H$. Choose $z\in V_H\setminus\{x\}$. (Such a node exists by Properties~1  and~3 of
  Theorem~\ref{thm:component}.)  Since $G$ is biconnected, there must
  exist a simple path in $G$ from $y$ to $z$ that does not visit $x$.
  Prepending $e$ to this path yields a simple path $\omega$ between $x$ and $z$ that
  exits $H$. By adding edges one at a time (e.g., as in Kruskal's algorithm) we can iteratively grow this path $\omega$ into a spanning tree $\ga$ of $G$. Thus, there exists at least one spanning tree
  of $G$ that includes $\omega$. By Theorem~\ref{thm:component}, every fair tree must restrict to a tree on $H$. Therefore, $\ga$ is a forbidden tree (see Definition \ref{def:fair-forbidden}). However, the uniform pmf $\mu_0$ defined in~\eqref{eq:mu-0} has every spanning tree of $\Ga_G$ in its support. Since the support of $\mu_0$ contains a forbidden tree, it cannot be  optimal, contradicting the assumption that $G$ is uniform.
\end{proof}

\section*{Acknowledgments}

The authors are grateful to Professors Marianne Korten and David Yetter at
Kansas State University for organizing an exceptional summer undergraduate
research program that led to several of the results presented in this paper.

\bibliographystyle{acm}
\bibliography{spt_modulus,pmodulus,myrefs}
\def\cprime{$'$}

\end{document}